\documentclass[a4paper,10pt]{amsart}

\usepackage{amsmath, amssymb, amscd, amsfonts, mathrsfs, epic, verbatim}
\usepackage[usenames, dvipsnames]{color}
\usepackage{hyperref}
\usepackage{enumerate}
\usepackage{enumitem}

\allowdisplaybreaks[4]
	
\def\name{P\'eter Ivanics}
\hypersetup{bookmarksopen=true}
\hypersetup{pdffitwindow=false}
\hypersetup{pdfstartview=FitH}
\hypersetup{pdfencoding=auto}						
\hypersetup{
		pdfauthor = {\name},
		pdfkeywords = {algebraic geometry, logarithmic connection, Fuchs equation, moduli space},
		pdftitle = {The locus of the representation of logarithmic connection with Fuchs equation},
		pdfsubject = {article},
		}
\newtheorem{prop}{Proposition}[section]
\newtheorem{lem}[prop]{Lemma}
\newtheorem{defn}[prop]{Definition}
\newtheorem{cor}[prop]{Corollary}
\newtheorem{thm}[prop]{Theorem}
\newtheorem{assn}[prop]{Assumption}
\newtheorem{rem}[prop]{Remark}
\newtheorem{fact}[prop]{Fact}	

\newtheorem{conj}[prop]{Conjecture}

\makeatletter
\newsavebox\myboxA
\newsavebox\myboxB
\newlength\mylenA
\newcommand\xoverline[2][0.75]{%
    \sbox{\myboxA}{$\m@th#2$}%
    \setbox\myboxB\null
    \ht\myboxB=\ht\myboxA%
    \dp\myboxB=\dp\myboxA%
    \wd\myboxB=#1\wd\myboxA
    \sbox\myboxB{$\m@th\overline{\copy\myboxB}$}
    \setlength\mylenA{\the\wd\myboxA}
    \addtolength\mylenA{-\the\wd\myboxB}%
    \ifdim\wd\myboxB<\wd\myboxA%
       \rlap{\hskip 0.5\mylenA\usebox\myboxB}{\usebox\myboxA}%
    \else
        \hskip -0.5\mylenA\rlap{\usebox\myboxA}{\hskip 0.5\mylenA\usebox\myboxB}%
    \fi}
\makeatother

\makeatletter
\newcommand{\kitem}[1][]{%
\item[\textbullet\, #1]\protected@edef\@currentlabel{#1}\ignorespaces%
}
\makeatother

\newcommand{\codim}{\mathop{\mathrm{codim}}\nolimits}				
\newcommand{\Sym}{\mathop{\mathrm{Sym}}}										
\newcommand{\CP}[1]{\mathbb{C}P^{#1}}    									  
\newcommand{\Mod}{{\mathcal M}}           									
\newcommand{\bigslant}[2]{{\raisebox{.2em}{$#1$}\left/\raisebox{-.2em}{$#2$}\right.}}	
\newcommand\comp[1]{\xoverline{#1}}													
\newcommand{\jj}{\bar{\jmath}}

\DeclareTextCommand{\textoneinferior}{PU}{\9040\201}				
\DeclareTextCommand{\texttwoinferior}{PU}{\9040\202}				
\DeclareTextCommand{\textzerosuperior}{PU}{\9040\160}				

\title[The locus of the representation of logarithmic connections]{The locus of the representation of logarithmic connections by Fuchsian equations}
\author{P\'eter Ivanics}

\begin{document}
\begin{abstract}
The generic element of the moduli space of logarithmic connections with parabolic points on holomorphic vector bundle over the Riemann sphere can be represented by a Fuchsian equation with some singularities and some apparent singularities. We analyze the case of rank $3$ vector bundle which leads to third order Fuchsian equation. We find coordinates on an open subset of the moduli space and we construct a non-trivial part of the moduli space by blowing up along a variety in a special case.
\end{abstract}
\maketitle

\section{Introduction}
\label{sec:intro}
Let $\Sigma$ be a compact Riemann surface with $n+1$ distinct fixed points $P=\left\{ t_0, \dots, t_n\right\}$. We will call these points \emph{parabolic points}. Let $I$ be the set of indices of $t_i \in P$. Fix a rank $m$ smooth vector bundle $\mathcal{V}$ over $\Sigma$ and consider a holomorphic vector bundle $E$ with underlying topological vector bundle $\mathcal{V}$. Furthermore consider the \emph{logarithmic connections} on $E$.

\begin{defn}
Let $\Sigma$ be a compact Riemann surface with a rank $m$ holomorphic vector bundle $E$. Let $U_{\Sigma}$ be an open chart in $\Sigma$, let $f:U_{\Sigma} \rightarrow \mathbb{C}$ be any holomorphic function and let $\sigma \in \Gamma(U_{\Sigma},E)$. A \emph{logarithmic connection} $D$ over $E$ with parabolic points $t_i \in P$ is a $\mathbb{C}$-linear map $ D: E \rightarrow E \otimes \Omega_{\Sigma}^1 (\log(P))$ where $\Omega_{\Sigma}^1 (\log(P))$ is the sheaf of meromorphic $1$-forms with at most first order poles at $P$ over $\Sigma$ which satisfies the Leibniz rule
\begin{equation*}
	D(\sigma f) = (D\sigma)f + \sigma \otimes \mathrm{d}f.
\end{equation*}
\end{defn}

We introduce a local chart $U$ in $\Sigma$ around a point $t\notin P$ and let $z$ be a coordinate on $U$.
A logarithmic connection on $U$ has the shape $D=\mathrm{d}+A(z) \mathrm{d}z$, where $A(z)$ is a matrix-valued holomorphic function. 

\noindent Let $U_i$ be a local chart in $\Sigma$ around a point $t_i\in P$ and let $z_i$ be a coordinate on $U_i$ such that $t_i=\{z_i=0\}$.
A logarithmic connection on $U_i$ has the shape $D=\mathrm{d}+\frac{\tilde{A}(z_i)}{z_i} \mathrm{d}z_i$, where $\tilde{A}(z_i)$ is a matrix-valued holomorphic function. We refer to the $\tilde{A}(t_i)$'s as residues and we fix their adjoint-orbits. We denote the eigenvalues of $\tilde{A}(t_i)$ by $\rho_{i,k}$, where $k=1, \dots, m$. The eigenvalues $\rho_{i,k}$ are called \emph{generic} if 
we choose az integer $1\leq l\leq m$ and any $K_i\subseteq \{1, \dots, m\}$ with $|K_i|=l$ $\forall i \in I$ then $\sum_{i \in I} \sum_{k\in K_i} \rho_{i,k} \notin \mathbb{Z}$.

\begin{defn}
A connection $D$ over $E$ is \emph{irreducible} if $E$ has no nontrivial $D$-invariant subbundles.
\end{defn}

Consider the pairs $(E,D)$ where $D$ is a logarithmic connection with fixed $\left\{\rho_{i,k}\right\}_{i=0,k=1}^{n,m}$. 
\begin{assn}
\label{ass:generic}
We suppose that the eigenvalues $\rho_{i,k}$ are generic.
\end{assn}
This implies by the residue theorem all connections are irreducible.

For the construction of moduli spaces one needs a technical condition called stability. It follows immediately from definitions that irreducible connections are stable. Hence under the Assumption~\ref{ass:generic} all pairs $(E,D)$ with the given eigenvalues $\rho_{i,k}$ are stable. Thus we refrain ourselves from spelling out the definition of stability. 

\begin{rem}
If we fix the eigenvalues $\rho_{i,k}$ then the rank and degree of underlying vector bundle $\mathcal{V}$ will be fixed by the residue theorem.
\end{rem}

Two stable connections $(E,D)$ and $(E',D')$ are equivalent if there exists a bundle isomorphism $\phi : E \rightarrow E'$ conjugating the connections. Denote by $\sim$ this equivalence relation on the set of stable connections.
Let $\Mod = \left\{ (E,D) \right\}\!/\!\sim$ be the \emph{moduli space of stable logarithmic connections} with fixed $\left\{\rho_{i,k}\right\}_{i=0,k=1}^{n,m}$.

If $\Sigma=\mathbb{CP}^1$, then the moduli space $\Mod$ is a $2N$-dimensional complex algebraic variety where
\begin{equation}
  \label{eq:dimension}
  N= \frac{1}{2}\dim_{\mathbb{C}} \Mod = \frac{m(m-1)}{2}(n-1)-m+1 \quad\mbox{(Remark~6.33. in \cite{put})} .
\end{equation}
Moreover, there exists a holomorphic symplectic structure on $\Mod$ \cite{inaba}.
The topology on the moduli space $\Mod$ refers to the Zariski topology.

It is known that the above connections can be represented by a \emph{Fuchsian equation} with singular locus in $P$ and with some \emph{apparent singularities} outside $P$ \cite{put}. 
\begin{defn}
A \emph{Fuchsian equation} is a linear homogeneous ordinary differential equation of order $m$ in a complex domain $U_{\Sigma}$. Its local form is the following
\begin{equation}
  \label{eq:gen_diffeq}
	w^{(m)}(z) + B_1(z) w^{(m-1)}(z) + \cdots + B_m(z) w(z) = 0,
\end{equation}
where the $B_k(z)$'s are functions with a pole of order $k$ at some points $u_l\in U_{\Sigma}$ ($k=1,\dots,m$, $l\in L$ with set of indices $L$).
\end{defn}

\begin{defn} 
\label{def:app_sing}
A point $q \in \CP1$ is an \emph{apparent singularity} of a differential equation \eqref{eq:gen_diffeq} if any of the (meromorphic) coefficient of Equation~\eqref{eq:gen_diffeq} has a pole but the fundamental system of solutions is analytic at this point.
\end{defn}

\noindent We will label parabolic points by $i$ and apparent singularities by $j$ throughout the paper. Moreover, we will label by $i$ or $j$ the expressions related to $t_i$ or $q_j$. Denote the set of apparent singularities by $Q$.

Dubrovin and Mazzocco dealt with Fuchsian equations of arbitrary order in \cite{dubrovin}. They introduced apparent singularities and some auxiliary parameters as Darboux coordinates on the symplectic space of Fuchsian system. 

Szab\'o specified the number of apparent singularities in \cite{szilard08} and he completely described  
the case of rank-$2$ vector bundle, which leads to a second order Fuchsian equation \cite{szilard13}. The present paper provides an alternative approach to determine Darboux coordinates. Szab\'o also describes the cases $q_{j_1}=q_{j_2}$ and $t_i=q_j$ for some indices by a blow up method.

The third order case is significantly more involved because the coefficient matrix of system of equations is not block diagonal, hence new techniques are needed for the analysis. 
The case of $4$ parabolic points (Theorem~\ref{th:main_V1}) is the smallest non-trivial example, where there exists a non-empty set over which the natural Darboux coordinates of the moduli space do not extend. This case differs from the previously studied cases (i.e. $q_{j_1}=q_{j_2}$ and $t_i=q_j$ cases) in \cite{szilard13}. Conjecture~\ref{cj:main_V2} describes how to construct a non-trivial part of the moduli space by blowing up and how to choose coordinates.

Iena and Leytem \cite{iena} point out a similar phenomenon in the Simpson moduli space of semi-stable sheaves. Namely, they blow up the moduli space along a singular subvariety and they get a closed subset of the moduli space of non-singular sheaves.

\subsection{The main results}
Before the main result we introduce some notations: let $\Delta$ be the subset where $q_{j_1}=q_{j_2}$ or $t_i=q_{j_3}$ for any $q_{j_1},q_{j_2},q_{j_3}\in Q$ and $t_i\in P$, let $N$ be the value in Equation~\eqref{eq:dimension} and let $\Sym(N)$ be the $N$-th symmetric group.

\begin{thm}
\label{th:main_V1}
Let $\Sigma$ be the Riemann sphere with $n$ parabolic points and let $E$ be a holomorphic vector bundle of rank $m=3$ with underlying topological vector bundle $\mathcal{V}$ over $\Sigma$.
Let $\Mod$ be the moduli space of logarithmic connections over $E$ with fixed generic eigenvalues $\left\{\rho_{i,k}\right\}_{i=0,k=1}^{n,3}$.
Then there exists a dense open subset $\Mod^0$ of $\Mod$ and a $\Sym(N)$-invariant affine subvariety $V$ in \mbox{$T^{*}(\mathbb{C} \setminus P)^N$} such that
\begin{equation*}
\Mod^0 = \bigslant{ T^{*} \left( \left( \mathbb{C} \setminus P \right)^N \setminus \Delta \right) \setminus V}{\Sym(N)}.
\end{equation*}
\end{thm}

\begin{cor}
We get a canonical coordinate system on $\Mod^0$ with coordinates $q_j$ on the base of $(\mathbb{C} \setminus P)^N / \Sym(N)$ and coordinates $p_j$ on the fiber over $q_j$.
\end{cor}

The following conjecture is an extension of the previous theorem.
\begin{conj}
\label{cj:main_V2}
Let $\Sigma$ be the Riemann sphere with $n$ parabolic points, fix generic eigenvalues $\left\{\rho_{i,k}\right\}_{i=0,k=1}^{n,3}$ and consider the moduli space of logarithmic connections $\Mod$ with eigenvalues $\rho_{i,k}$.
Let $V$ be the subvariety from Theorem~\ref{th:main_V1}. Then there exists a $\Sym(N)$ invariant affine subvariety $\hat{V}\neq \emptyset$ in $T^{*}(\mathbb{C} \setminus P)^N$, which intersects $V$ generically transversely and there exists an open subset of $\left( V \cap \hat{V} \right)$ (denoted by $(V \cap \hat{V})^0$) 
such that the following blow-up is a subset of the moduli space:
\begin{equation*}
	\Mod^1 := \mathrm{Blow}_{\left(V \cap \hat{V}\right)^0} \left(\bigslant{\left(T^{*} \left( \mathbb{C}\setminus P \right)^N \setminus \Delta \right)}{\Sym(N)}\right) \setminus \overset{\sim}{V} \subseteq \Mod,
\end{equation*}
where $\overset{\sim}{V}$ is the proper transform of $V$. (Obviously, $\Mod^0$ is a subset of $\Mod^1$.)
\end{conj}
We will show a general numerical example in the case $n=3$ where this conjecture holds.
Another natural question is what is the situation with higher rank ($m>3$) vector bundle. The conjecture is that there exists two discriminant varieties $V$ and $\hat{V}$ such that a result similar to the above holds.
The computations in this article were made using \emph{Wolfram Mathematica~11.0}.

\subsection{Outline of the paper}
The next section contains the initial statement and equation which is the starting point of our analysis of moduli space. The section also describes some well-known tools which are useful for the computations. 
In Section~\ref{sec:third_order}, we write the equations explicitly in the case of rank-$3$ vector bundles and we collect the conditions which guarantee the appearance of parabolic points and apparent singularities. 
In Section~\ref{sec:n=3}, we analyze the case $n=3$. We construct the discriminant variety $V$ and make a numerical example. After that we show a blow up procedure in support of Conjecture~\ref{cj:main_V2}. 
In Section~\ref{sec:m=3}, we compute the variety $V$ in the case of arbitrary many parabolic points and we prove Theorem~\ref{th:main_V1}.

\bigskip
\noindent \textbf{Acknowledgments:} The author acknowledges support of NKFIH through the \'Elvonal (Frontier) program, with grant KKP126683.
The author wants to thank Szil\'ard Szab\'o who offered the research topic, gave some useful suggestions and paid attention to the writing of the paper. Thanks to Andr\'as Stipsicz for many useful comments and encouragement.

\section{Preparatory material}
\label{sec:prep}
The starting point of this paper is a theorem of Katz \cite{katz} extended by Dubrovin and Mazzocco in \cite{dubrovin}. Here we summarize and apply their results which are relevant to us.

Let $\Sigma$ be a compact Riemann surface with $t_0, \dots, t_n \in \Sigma$, let $E$ be a holomorphic bundle on $\Sigma$, let $U_{\Sigma}$ be an open chart in $\Sigma$ and let $D$ be a logarithmic connection over $E$ with logarithmic points $t_0, \dots, t_n$. 
Consider the sheaf $\mathcal{S}$ on $\Sigma \setminus \{t_0,\dots,t_n\}$ given by $L(U_{\Sigma}):=\{e\in \Gamma(U_{\Sigma},E) | D(e)=0\}$. 
This gives a functor $\Psi$ from the category of regular connections on $\Sigma$ to the category of local systems on $\Sigma$, i. e. the locally constant sheaves of $\mathbb{C}$-vector spaces.
	
Let $\tau: [0, 1] \rightarrow \Sigma$ be a path in $\Sigma$.
Let $L$ be a local system on $\Sigma$. Then $\tau^* L$ is a local system on $[0, 1]$. The triviality
of this local system yields an isomorphism $(\tau^* L)_0 \rightarrow (\tau^* L)_1$. The two stalks
 $(\tau^* L)_0$, $(\tau^* L)_1$ are canonically identified with $L_{\tau(0)}$ and $L_{\tau(1)}$. Thus we find
an isomorphism $L_{\tau(0)} \rightarrow L_{\tau(1)}$ induced by $\tau$. Let $b$ be a base point for $\Sigma$ and let $\pi_1$ denote the fundamental group of $\Sigma \setminus \{t_0,\dots,t_n\}$ with respect to this base point. Let $F$ denote the stalk $L_b$. Then for any closed path $\tau$ through $b$ we find an automorphism of $F$. In this way we have associated to $L$ a representation $\rho_L : \pi_1 \rightarrow \mathsf{GL}(F)$ of the fundamental group. More generally, we get a functor $\rho$ which associates a representation of the fundamental group of $\Sigma \setminus \{t_0,\dots,t_n\}$ to a local system on $\Sigma$.

\begin{defn}
The composition  
\begin{equation*}
	\rho \circ \Psi: \mathsf{Log.~conn.}(\Sigma, t_0,\dots,t_n) \longrightarrow \mathsf{Hom}\left(\pi_1\left(\Sigma \setminus \{t_0,\dots,t_n\}\right),\mathsf{GL}(F)\right)
\end{equation*}
gives a functor which associates a so-called \emph{monodromy representation} to a regular connection.\cite{put}
\end{defn}

\begin{defn}
Let a (Fuchsian) differential equation over an open, connected set $U$ in the complex plane. Let $\pi_1 (U)$ the fundamental group of $U$. The \emph{monodromy representation of the differential equation} is the linear representation of $\pi_1 (U)$. 
\end{defn}

\begin{defn}
We say that the logarithmic connection $D$ over $E$ is \emph{represented} by a Fuchsian equation if their monodromy representations are the same up to global conjugation.
\end{defn}

\begin{thm}
  \label{th:katz}
	Fix a Riemann surface $\Sigma$, a rank $m$ holomorphic vector bundle $E$ and set of points $P$ with fixed eigenvalues $\rho_{i,k}$ of residues of logarithmic connection at $t_i\in P$ ($i=0,\dots,n$, $k=1, \dots, m$). 
	Then every element of the above moduli space $\Mod$ can be represented by an order~$m$ Fuchsian equation~\eqref{eq:gen_diffeq} which has singularities at the points $t_i \in P$ with exponents $\rho_{i,k} \pmod{2\pi \mathbb{Z}}$ and has apparent singularities at points of some set $\left\{q_1, \dots, q_g \right\}=Q$ for some value~$g$.

Moreover, this correspondence gives Darboux coordinates $(q_j, p_j)_{j=1}^{g}$ on the moduli space $\Mod$ where $p_j$'s derive from the coefficients of the representing Fuchsian equation.
\end{thm}
Katz does not specify the value of $g$, Dubrovin and Mazzocco specified $g=N$ from \ref{eq:dimension} but does not describe the variety when the Fuchsian representation is not canonical and depends on some choices. We only use the fact that there exists a Fuchsian representation and Darboux coordinates. We will analyze the moduli space with different techniques.

In the case $\Sigma=\CP1$ Szab\'o specified the value $g$ as the number $g=N=\frac{1}{2}\dim_{\mathbb{C}} \Mod$ in \cite{szilard08}. Moreover, he proved that the eigenvalues $\rho_{i,k}$ of residues and exponents $\rho_{i,k}$ of singularities are equal in generic case (i. e. when $q_{j_1} \neq q_{j_2}$).
In the case of second order Fuchsian equation \cite{szilard13} he assigned the auxiliary parameters $p_j$ ($j=1,\dots,N$) from the coefficient $B_2(z)$. We will follow this method in our article and it will turn out that the auxiliary parameters of the third order system also come from $B_2(z)$.

\subsection{The Frobenius method}
\label{sec:frobenius}
We will examine Equation~\eqref{eq:gen_diffeq} in the case $\Sigma=\CP1$ and $m=3$ with the Frobenius method. For this reason, we provide a short summary about this method. Write the coefficients of Equation~\eqref{eq:gen_diffeq} as quotients of two polynomials. The roots of the denominator are $t_i\in P$ and $q_j \in Q$ for all $i$ and $j$. Let $z$ be the standard coordinate on the chart of $\CP1$ away from $[1:0]$.
\begin{subequations}
  \label{eq:gen_ghipsi}
  \begin{align}
    B_1(z) & = \frac{G(z)}{\psi(z)},\; B_2(z) = \frac{H(z)}{\psi^2 (z)}, \dots, B_m(z) = \frac{K(z)}{\psi^m(z)}, \\
    \psi(z) & = \prod_{i=1}^n (z-t_i) \prod_{j=1}^N (z-q_j).
  \end{align}
\end{subequations}

\begin{defn}
\label{def:indicial_poly}
The \emph{indicial polynomial} of Equation~\eqref{eq:gen_diffeq} is the following:
\begin{equation}
  \label{eq:gen_indicial}
  \rho_i(\rho_i-1) \cdots (\rho_i-m+1) + G_0^{t_i} \rho_i(\rho_i-1) \cdots (\rho_i-m+2) + \dots + K_0^{t_i} = 0,
\end{equation}
where $G_0^{t_i}$, \dots, $K_0^{t_i}$ are the $-1\textsuperscript{st}$, $\dots$, $-m\textsuperscript{th}$ elements of the Laurent series of $B_1(z)$, \dots, $B_m(z)$ 
 at the point $t_i$. The roots $\rho_{i,1}$, \dots, $\rho_{i,m}$ of indicial polynomial are called \emph{exponents} which belong to singularity $t_i$.
\end{defn}

\begin{thm}
  \label{th:frobenius}
	Let the exponents $\rho_{i,k}$ ($k=1,\dots,m$) be generic (see Assumption~\ref{ass:generic}).
If in Equation~\eqref{eq:gen_diffeq} $G(z)$, $H(z)$,\dots, $K(z)$ are holomorphic functions near the neighborhood of $t_i$, then the solutions $w(z)$ have the following shape in an angular sector with vertex $t_i$:
\begin{equation*}
	w_{i,k}(z)=C_{i,k} (z) \cdot (z-t_i)^{\rho_{i,k}},
\end{equation*}
where $C_{i,k}(z)\neq 0$ are holomorphic functions $(k=1,\dots,m)$.
\end{thm}
Similar theorem concerns the apparent singularities $q_j$, except that the exponents belonging to $q_j$ must be natural numbers. Therefore the exponents belonging to $q_j$ are not generic in the sense than $\rho_{i,k}$'s and we need tools to handle them. The next facts concern the case $m=3$ but analogous statements hold for any order of differential equation at any (apparent) singularity.

\begin{fact}
\label{fact:logterms}
Let $\rho_{i,1}$, $\rho_{i,2}$ and $\rho_{i,3}$ be the exponents belonging to an (apparent) singularity $u_i$. Let indices $k_1, k_2,k_3 \in\{1,2,3\}$ and let $C_{k_j}$ ($j=1,2,3$) be non-zero holomorphic functions. The following hold about the solutions of~\eqref{eq:gen_diffeq}:
\begin{enumerate}
	
	\item If $\rho_{i, k_1} - \rho_{i, k_2} \notin \mathbb{Z}$ for all $k_1, k_2$ then there exists a fundamental system of solutions of the form 
	\begin{equation*}
	w_ {i, k} = z^{\rho_{i, k}} \cdot C_k(z),
	\end{equation*}
	where $C_k(z)\neq 0$ are holomorphic functions $(k=1,2,3)$.

	\item If  $\rho_{i, k_1} - \rho_{i, k_2}\in \mathbb{Z}$, $\rho_{i, k_1} - \rho_{i, k_3} \notin \mathbb{Z}$ and $\rho_{i, k_1} \leq \rho_{i, k_2}$ for any $k_1, k_2, k_3$ then there exists a fundamental system of solutions of the form 
	\begin{align*}
	w_{i, k_1} =& z^{\rho_{i, k_1}} \cdot C_{k_1}(z), \\
	w_{i, k_2} =& z^{\rho_{i, k_2}} \cdot C_{k_2}(z) + \kappa_{k_1,k_2}^i w_{i, k_1} \log(z-u_i), \\
	w_{i, k_3} =& z^{\rho_{i, k_3}} \cdot C_{k_3}(z),
	\end{align*}
	where $\kappa_{k_1,k_2}^i \in \mathbb{C}$ is constant. 
	\label{it:logterm1}
	
	\item If  $\rho_{i, k_1} - \rho_{i, k_2}\in \mathbb{Z}$, $\rho_{i, k_2} - \rho_{i, k_3} \in \mathbb{Z}$ and $\rho_{i,k_1}\leq \rho_{i,k_2}\leq \rho_{i,k_3}$ for any $k_1, k_2, k_3$ then there exists a fundamental system of solutions of the form 
	\begin{align*}
	w_{i, k_1} =& z^{\rho_{i, k_1}} \cdot C_{k_1}(z), \\
	w_{i, k_2} =& z^{\rho_{i, k_2}} \cdot C_{k_2}(z) + \kappa_{k_1,k_2}^i w_{i, k_1} \log(z-u_i), \\
	w_{i, k_3} =& z^{\rho_{i, k_3}} \cdot C_{k_3}(z) + \kappa_{k_1,k_3}^i w_{i, k_1} \log(z-u_i) + \\
							& + \kappa_{k_2,k_3}^i w_{i, k_2} \log(z-u_i),
	\end{align*}
	where $\kappa_{k_{j_1},k_{j_2}}^i \in \mathbb{C}$ ($j_1,j_2\in \{1,2,3\}$) are constant. \label{it:logterm2}
	
		\item If  $\rho_{i, k_1} = \rho_{i, k_2}$ for any $k_1, k_2 \in \{1,2,3\}$ then the fundamental system of solutions is the same as in the case $\rho_{i, k_1} - \rho_{i, k_2}\in \mathbb{Z}$ except that $\kappa_{k_1,k_2}^i \neq 0$. \label{it:nonzero_logterm}

\end{enumerate}
\end{fact}

\begin{prop}
\label{prop:app_sing_exp}
At an apparent singularity $q_j$ the set of exponents $\left\{\rho_{j,k}\right\}_1^m$ are pairwise distinct non-negative integers.
\end{prop} \qed

Since $\rho_{i,k}$ are fixed for every $t_i$ in the moduli space $\Mod$, the values of $G_0^{t_i}$, $H_0^{t_i}$,$\dots$ can be computed from the indicial polynomial. Moreover, $G_0^{q_j}$, $H_0^{q_j}$,$\dots$ will be computed from the indicial polynomial at $q_j$.

It is well known that above conditions are not linearly independent, the \emph{Fuchs' relation} holds for all exponents of all singularities due to residue theorem \cite{anosov}, namely:
\begin{equation}
  \label{eq:fuchs}
  \sum_{i=0, k=1}^{n,m} \rho_{i,k} = \frac{(n-2)m(m-1)}{2}.
\end{equation}

\subsection{Confluent Vandermonde matrices}
\label{sec:vandermonde}
A \emph{confluent Vandermonde matrix} is a certain generalization of Vandermonde matrix. Let $n$ be a natural number, let $X=\left\{x_1,x_2,\dots,x_m\right\}$ be a set of parameters and let $\lambda$ be a length $m$ integer partition of $n$ such that $\lambda=(n_1,n_2,\dots,n_m) \vdash  n$.
\begin{defn}
Define the $n \times n$ confluent Vandermonde matrix $M$ corresponding to $X$ and $\lambda$ in the following way.
For all indices $i\in\{1,\dots,m\}$ there is a row of $M$ which contains the consecutive powers of the parameter $x_i$ from exponent zero to $n-1$. We say this is a Vandermonde-like row:
\begin{equation*}
	1,x_i,x_i^2, x_i^3, \dots, x_i^{n-1}.
\end{equation*}
If any of $n_i>1$ then $M$ contains rows which are $1\textsuperscript{st}$, $2\textsuperscript{nd}$, $\dots$, $(n_i-1)\textsuperscript{th}$ derivative elementwise of above Vandermonde-like row with respect to $x_i$ and divided by $1!$, $2!$, $\dots$, $(n_i-1)!$ respectively. 
\end{defn}
\begin{defn}
\label{def:st_seq}

In this article we require that the rows follow each other by the index of parameters, within that by order of derivatives. (But we do not require that if $i<j$ then $n_i\geq n_j$.)  We will refer to this order as \emph{standard sequence}.
\end{defn}

\begin{prop} \cite{horn}
\label{prop:conf_vand}
Let $M$ be the $n \times n$ confluent Vandermonde matrix with the set of parameters $X$ (with cardinality $m$) and $\lambda \vdash n$.
Let $n_i$ be the cardinality of rows which contain parameter $x_i$ ($n_i-1$ is the maximum order of derivatives with respect to $x_i$'s). The determinant of such a confluent Vandermonde matrix is the following:
\begin{equation}
  \label{eq:conf_vand}
  \det(M)= \prod_{1\leq i<j \leq m} (x_j-x_i)^{n_i \cdot n_j}=:\mathrm{ConfVand} \left( \left\{ x_i^{(n_i-1)} \right\}_{i=1}^m \right),
\end{equation} 
where we introduce a function $\mathrm{ConfVand}$ which gives the determinant of a confluent Vandermonde matrix with parameters $x_i$ and order of derivatives $n_i-1$ \cite{ha-gibson}.
\end{prop}

\section{The third order system}
\label{sec:third_order}
First of all, fix the notation in the case of rank-$3$ vector bundles (i. e. $m=3$). The Riemann sphere $\CP1$ has $n+1$ fixed points $t_i$ (the parabolic points) with fixed exponents $\rho_{i,k}$ ($i=0,\dots n$, $k=1,2,3$) and $N=3n-5$ points (the apparent singularities of Fuchsian equation) according to Equation~\eqref{eq:dimension}. They are:
$$
 P=\left\{ t_0, \dots, t_n\right\}, \hspace{10mm} Q=\left\{ q_1, \dots, q_{3n-5}\right\},
$$
For simplicity, choose $t_0$ at infinity.

Now, the differential equation~\eqref{eq:gen_diffeq} has the following shape:
\begin{equation*}
  w'''(z) + B_1(z) w''(z) + B_2(z) w'(z) + B_3(z)  w(z) = 0.
\end{equation*}
Write the coefficients $B_i(z)$ as a quotient of two polynomials as in Equations~\eqref{eq:gen_ghipsi}.
The denominators of these fractions are powers of the following polynomial:
\begin{equation}
	\psi(z) = \prod_{i=1}^n (z-t_i) \prod_{j=1}^{3n-5} (z-q_j).
	\label{eq:psi}
\end{equation}
The polynomials in the numerators are:
\begin{subequations}
\label{eq:poly}
\begin{align}
\label{eq:poly_gz}
  B_1 (z) \psi(z) = G(z) & = G_0 + \dots+ G_{4n-6} z^{4n-6}, \\
  B_2 (z) {\psi(z)}^2 = H(z) & = H_0 + \dots+ H_{8n-12} z^{8n-12}, \\
  B_3 (z) {\psi(z)}^3 = I(z) & = I_0 + \dots+ I_{12n-18} z^{12n-18}.
\end{align}
\end{subequations}
Namely, the differential equation takes the following shape:
\begin{equation}
  \label{eq:diffeq}
  w'''(z) + \frac{G(z)}{\psi(z)} w''(z) + \frac{H(z)}{\psi^2(z)} w'(z) + \frac{I(z)}{\psi^3(z)} w(z) = 0.
\end{equation}
Equation~\eqref{eq:diffeq} contains $24n-33$ unknown coefficients. 

\subsubsection{Conditions from parabolic points}
Frobenius method provides a solution of Equation~\eqref{eq:diffeq} near $z=t_i$ as the form $w(z) = \sum_{l=0}^{\infty} a_l (z-t_i)^{l+\rho_i}$ and get the \emph{indicial equation} ($i=0, \dots, n$):
\begin{equation}
	\label{eq:indicial}
  \rho_i(\rho_i-1)(\rho_i-2) + G_0^{t_i} \rho_i(\rho_i-1) + H_0^{t_i} \rho_i + I_0^{t_i} = 0,
\end{equation}
where $G_0^{t_i}$, $H_0^{t_i}$, $I_0^{t_i}$ denote the $-1\textsuperscript{st}$, $-2\textsuperscript{nd}$ and $-3\textsuperscript{rd}$ $z$ coefficients of Laurent series of the coefficient of Equation~\eqref{eq:diffeq} at $t_i$ ($i=1, \dots, n$).
\begin{subequations}
  \label{eq:ghi_ti0}
  \begin{align}
    G_0^{t_i} & = \lim_{z \rightarrow t_i} (z-t_i)\frac{G(z)}{\psi (z)} = \frac{G_0+G_1 t_i+\dots+G_{4n-6}t_i^{4n-6}}{\prod_{k=1, k\neq i}^n (t_i-t_k) \prod_{l=1}^{3n-5} (t_i-q_l)}, \\
    H_0^{t_i} & = \lim_{z \rightarrow t_i} (z-t_i)^2 \frac{H(z)}{\psi^2 (z)} = \frac{H_0+H_1 t_i+\dots+H_{8n-12}t_i^{8n-12}}{\prod_{k=1, k\neq i}^n (t_i-t_k)^2 \prod_{l=1}^{3n-5} (t_i-q_l)^2}, \\
    I_0^{t_i} & = \lim_{z \rightarrow t_i} (z-t_i)^3 \frac{I(z)}{\psi^3 (z)} = \frac{I_0+I_1 t_i+\dots+I_{12n-18}t_i^{12n-18}}{\prod_{k=1, k\neq i}^n (t_i-t_k)^3 \prod_{l=1}^{3n-5} (t_i-q_l)^3}.
  \end{align}
\end{subequations}
Similar limits at $t_0=\infty$ are given by
\begin{subequations}
\label{eq:ghi_t00}
  \begin{align}
    G_0^{t_0} & = G_{4n-6}, \\
    H_0^{t_0} & = H_{8n-12}, \\
    I_0^{t_0} & = I_{12n-18}.
  \end{align}
\end{subequations}

The indicial equation~\eqref{eq:indicial} reads as
$$
  \rho_i^3 + (G_0^{t_i}-3) \rho_i^2 + (H_0^{t_i}-G_0^{t_i}+2)\rho_i + I_0^{t_i} = 0.
$$
Denote the (fixed) roots by $\rho_{i,k}$ ($k=1,2,3$). Vieta's formulas say:
\begin{subequations}
  \label{eq:vieta}
  \begin{align}
    \rho_{i,1} + \rho_{i,2} + \rho_{i,3} & = 3-G_0^{t_i}, \\
    \rho_{i,1} \rho_{i,2} + \rho_{i,1} \rho_{i,3} + \rho_{i,2} \rho_{i,3} & = H_0^{t_i} - G_0^{t_i} + 2, \\
    \rho_{i,1} \rho_{i,2} \rho_{i,3} & = -I_0^{t_i}.
  \end{align}
\end{subequations}

Express $G_0^{t_i}$, $H_0^{t_i}$, $I_0^{t_i}$ and substitute to Equations~\eqref{eq:ghi_ti0} and \eqref{eq:ghi_t00}. 
We get the following system of linear equations in coefficients of $G$, $H$ and $I$ as variables ($i=1,\dots, n$):

{\allowdisplaybreaks
\begin{subequations}
  \begin{align}
    \label{eq:gt0}
    G_{4n-6} =& 3-(\rho_{0,1} + \rho_{0,2} + \rho_{0,3}), \\
    \label{eq:gti}
    \sum_{k=0}^{4n-6} t_i^k G_k =& \left[ 3-(\rho_{i,1} + \rho_{i,2} + \rho_{i,3}) \right] \prod_{k=1, k\neq i}^n (t_i-t_k) \prod_{l=1}^{3n-5} (t_i-q_l), \\
    \label{eq:ht0}
    H_{8n-12} =& \rho_{0,1} \rho_{0,2} + \rho_{0,1} \rho_{0,3} + \rho_{0,2} \rho_{0,3} -(\rho_{0,1} + \rho_{0,2} + \rho_{0,3}) +1, \\
		\label{eq:hti}
		\begin{split}
				\sum_{k=0}^{8n-12} t_i^k H_k =& \left[ \rho_{i,1} \rho_{i,2} + \rho_{i,1} \rho_{i,3} + 
						\rho_{i,2} \rho_{i,3} -(\rho_{i,1} + \rho_{i,2} + \rho_{i,3}) +1 \right] \cdot \\
				& \cdot \prod_{k=1, k\neq i}^n (t_i-t_k)^2 \prod_{l=1}^{3n-5} (t_i-q_l)^2,
		\end{split} \\
    \label{eq:it0}
    I_{12n-18} =& -\rho_{0,1} \rho_{0,2} \rho_{0,3}, \\
    \label{eq:iti}
    \sum_{k=0}^{12n-18} t_i^k I_k =& \left( -\rho_{i,1} \rho_{i,2} \rho_{i,3} \right) \prod_{k=1, k\neq i}^n (t_i-t_k)^3 \prod_{l=1}^{3n-5} (t_i-q_l)^3.
  \end{align}
\end{subequations}}

\subsection{Conditions from apparent singularities}
\label{sec:app_sing}
We treat the case of apparent singularities similarly. Compute $G_0^{q_j}$, $H_0^{q_j}$, $I_0^{q_j}$ for all $q_j$ ($j=1, \dots, 3n-5$):
  \begin{align*}
    G_0^{q_j} & = \lim_{z \rightarrow q_j} (z-q_j)\frac{G(z)}{\psi (z)} = \frac{G_0+G_1 q_j+\dots+G_{4n-6}q_j^{4n-6}}{\prod_{k=1}^n (q_j-t_k) \prod_{l=1, l\neq j}^{3n-5} (q_j-q_l)}, \\
    H_0^{q_j} & = \lim_{z \rightarrow q_j} (z-q_j)^2 \frac{H(z)}{\psi^2 (z)} = \frac{H_0+H_1 q_j+\dots+H_{8n-12}q_j^{8n-12}}{\prod_{k=1}^n (q_j-t_k)^2 \prod_{l=1, l\neq j}^{3n-5} (q_j-q_l)^2}, \\
    I_0^{q_j} & = \lim_{z \rightarrow q_j} (z-q_j)^3 \frac{I(z)}{\psi^3 (z)} = \frac{I_0+I_1 q_j+\dots+I_{12n-18}q_j^{12n-18}}{\prod_{k=1}^n (q_j-t_k)^3 \prod_{l=1, l\neq j}^{3n-5} (q_j-q_l)^3}.
  \end{align*}
Otherwise, we get a system of linear equations from indicial equations and Vieta's formulas, similarly as in Equation~\eqref{eq:vieta} (switch index $i$ to $j$ and $t_i$ to $q_j$). 

The exponents of apparent singularities are not fixed, but they are natural numbers since Definition~\ref{def:app_sing}. Hence case of (\ref{it:logterm2}) or case of (\ref{it:nonzero_logterm}) in Fact~\ref{fact:logterms} satisfied. We will can choose the values of $\rho_{j,k}$ ($j=1, \dots, 3n-5$; $k=1,2,3$).

The roots of the indicial polynomial cannot be equal because it leads to a nonzero logarithmic term according to the case of~(\ref{it:nonzero_logterm}) in Fact~\ref{fact:logterms}. 

\begin{defn}
The \emph{defect} of an apparent singularity $q_j$ is
\begin{equation}
	\delta:=\sum_{k=1}^{m} \rho_{j,k} - \frac{m (m-1)}{2}.
\end{equation}
\end{defn}
The defect measures the complexity of an apparent singularity $q_j$. The defect vanishes at a smooth point.

Indeed the set of exponents with defect zero is $\rho_{j,1} = 0$, $\rho_{j,2} = 1$ and $\rho_{j,3} = 2$ for all $j$. But then $G_0^{q_j}=0$, $H_0^{q_j}=0$ and $I_0^{q_j}=0$, hence the differential equation has no apparent singularity at $q_j$. Thus we exclude this possibility.

In view of Proposition~\ref{prop:app_sing_exp} the next simplest possibility is $\delta=1$.
\begin{assn}
The defect at all apparent singularities $q_j$ is equal to $1$ which means:
\begin{equation}
	\rho_{j,1} = 0, \rho_{j,2} = 1, \rho_{j,3} = 3 \mbox{~for all~} j.
\label{eq:app_sing_rho}
\end{equation}
\end{assn}

Then $G_0^{q_j}=-1$, $H_0^{q_j}=0$ and $I_0^{q_j}=0$ and we get the following system of linear equations for the coefficients ($j=1,\dots, 3n-5$):
\begin{subequations}
  \begin{align}
    \label{eq:gqj}
    \sum_{k=0}^{4n-6} q_j^k G_k & = -\prod_{k=1}^n (q_j-t_k) \prod_{l=1, l\neq j}^{3n-5} (q_j-q_l), \\
    \label{eq:hqj}
    \sum_{k=0}^{8n-12} q_j^k H_k & = 0, \\
    \label{eq:iqj}
    \sum_{k=0}^{12n-18} q_j^k I_k & = 0.
  \end{align}
\end{subequations}

\subsubsection{Conditions for the vanishing of logarithmic terms}
The above chosen $\rho_{j,k}$ exponents may cause the appearance of logarithmic terms with coefficients $\kappa_{k_1,k_2}^j$ in the solution of the differential equation in~\eqref{eq:diffeq} according to the case~(\ref{it:logterm2}) in Fact~\ref{fact:logterms}. The indices $(k_1,k_2)$ refer to the exponents $\rho_{j,k_1}$ and $\rho_{j,k_2}$.

For the existence of apparent singularities we need to exclude logarithmic terms; therefore, we must construct conditions for the vanishing of these terms. Note that the method in [Lemma 2.]\cite{szilard13} does not work here.

A fundamental system of solutions of Equation~\eqref{eq:diffeq} near $q_j$ based on the Frobenius method with the above values of exponents are the following:
\begin{subequations}
\label{eq:sol_total}
\begin{align}
		\begin{split}
    w_1(z-q_j) =& a_0 +a_1(z-q_j)+a_2(z-q_j)^2+\dots+ \\
				&+\kappa_{1,2}^j \log(z-q_j) \cdot w_2(z-q_j) + \kappa_{1,3}^j \log(z-q_j) \cdot w_3(z-q_j), 
		\end{split} \\
    w_2(z-q_j) =& b_1(z-q_j)+b_2(z-q_j)^2+\dots+ \kappa_{2,3}^j \log(z-q_j) \cdot w_3(z-q_j), \\
    w_3(z-q_j) =& c_3(z-q_j)^3+c_4(z-q_j)^4+\dots,
\end{align}
\end{subequations}
where $a_0\neq 0$, $b_1\neq 0$ and $c_3\neq 0$.

\begin{lem}
 The vanishing of logarithmic terms is equivalent to the following equations.
\begin{subequations}
\label{eq:log_vanishing}
\begin{align} 
  \frac{G(z)}{\psi(z)} & = -\frac{1}{z-q_j}+\left(-\frac{b_2}{b_1}-\frac{2 c_4}{c_3}\right)+\mathcal{O}(z-q_j), \\
  \frac{H(z)}{(\psi(z))^2} & = \frac{2 b_2}{b_1} \frac{1}{z-q_j}+\left(\frac{2a_1 b_2 -2 a_2 b_1}{a_0 b_1}+\frac{4 b_1 b_2 c_4-2b_2^2 c_3}{b_1^2 c_3}\right)+\mathcal{O}(z-q_j), \\
  \frac{I(z)}{(\psi(z))^3} & = \frac{2 a_2 b_1-2 a_1 b_2}{a_0 b_1} \frac{1}{z-q_j} + \frac{2 \left(a_2 b_1-a_1 b_2\right) \left(2 b_1 c_4 -b_2 c_3\right)}{a_0 b_1^2 c_3}+\mathcal{O}(z-q_j).
\end{align}
\end{subequations}
\end{lem}

\begin{proof}
Write the Wronskian from  the three functions of Equations~\eqref{eq:sol_total} and a function $w(z)$ as a general solution of Equation~\eqref{eq:diffeq}. Convert the Wronskian to the shape of \eqref{eq:diffeq} with dividing by coefficient of $w'''(z)$. Note that it is enough to compute terms with non-positive exponents in $(z-q_j)$ in the Wronskian, because the higher order terms do not give contributions to coefficients of lower $(z-q_j)$-powers. Consider the Laurent series expansions of the coefficients about $q_j$.  

\begin{footnotesize}
\begin{align*}
  \frac{G(z)}{\psi(z)} =& -\frac{1}{z-q_j}+\left(\frac{2 b_1 \kappa _{1,2}^j}{3 a_0}-\frac{b_2}{b_1}-\frac{2 c_4}{c_3}\right)+\mathcal{O}(z-q_j), \\
  \frac{H(z)}{(\psi(z))^2} =& \left( \frac{2 b_2}{b_1}-\frac{2 b_1 \kappa _{1,2}^j}{a_0}\right) \frac{1}{z-q_j} -\frac{2}{3 a_0^2 b_1^2 c_3} \left(3 a_0 b_1^3 c_4 \kappa _{1,2}^j-a_0 b_2 b_1^2 c_3 \kappa _{1,2}^j+9 a_0^2 b_1 c_3^2 \kappa _{2,3}^j\right.+\\
	&+\left.2 b_1^4 c_3 (\kappa _{1,2}^j)^2+3 a_0 a_2 b_1^2 c_3-3 a_0 a_1 b_2 b_1 c_3-6 a_0^2 b_2 b_1 c_4+3 a_0^2 b_2^2 c_3\right)	+\mathcal{O}(z-q_j), \\
  \frac{I(z)}{(\psi(z))^3} =& \frac{2 b_1 \kappa_{1,2}^j}{a_0} \frac{1}{(z-q_j)^2} + \frac{2 \left(3 a_0 b_1^2 c_4 \kappa _{1,2}^j+2 b_1^3 c_3 (\kappa _{1,2}^j)^2+3 a_0 a_2 b_1 c_3-3 a_0 a_1 b_2 c_3\right)}{3 a_0^2 b_1 c_3} \frac{1}{z-q_j} - \\
	& -\frac{1}{9 a_0^3 b_1^2 c_3^2} \left(-15 a_0 b_1^4 c_3 c_4 (\kappa _{1,2}^j)^2+3 a_0 b_2 b_1^3 c_3^2 (\kappa _{1,2}^j)^2-18 a_0 a_2 b_1^3 c_3^2 \kappa _{1,2}^j+36 a_0^2 b_1^3 c_4^2 \kappa _{1,2}^j\right.+ \\
	& + \left. 18 a_0 a_1 b_2 b_1^2 c_3^2 \kappa _{1,2}^j+54 a_0^2 b_3 b_1^2 c_3^2 \kappa _{1,2}^j-30 a_0^2 b_2 b_1^2 c_3 c_4 \kappa _{1,2}^j+54 a_0^2 b_1^2 c_3^3 \kappa _{1,3}^j\right.+\\
	&+\left. 63 a_0^2 b_1^2 c_3^3 \kappa _{1,2}^j \kappa _{2,3}^j-18 a_0^2 b_2^2 b_1 c_3^2 \kappa _{1,2}^j-54 a_0^2 a_1 b_1 c_3^3 \kappa _{2,3}^j\right.+ \\
	&+\left. 54 a_0^2 b_1^2 c_3^3 \kappa _{1,2}^j \kappa _{2,3}^j \log (z-q_j) -8 b_1^5 c_3^2 (\kappa _{1,2}^j)^3-36 a_0^2 a_2 b_1^2 c_3 c_4+18 a_0^2 a_2 b_2 b_1 c_3^2\right. +\\
	&+\left. 36 a_0^2 a_1 b_2 b_1 c_3 c_4-18 a_0^2 a_1 b_2^2 c_3^2\right) +\mathcal{O}(z-q_j).
\end{align*}
\end{footnotesize}

If Equations~\eqref{eq:log_vanishing} hold then a routine check using $a_0\neq 0$, $b_1\neq 0$, $c_3\neq 0$ shows
$\kappa^j_{k_1,k_2}=0$ for all $k_1,k_2 \in \{1,2,3\}$, hence there are no logarithmic terms in Equation~\eqref{eq:sol_total}.

In the converse direction we choose $\kappa^j_{k_1,k_2}$ for $0$ in Equations~\eqref{eq:sol_total} and after that we make a Wronskian in the previous way. After the conversion and series expansions we get Equations~\eqref{eq:log_vanishing}.
\end{proof}

Take the series expansions by definition of the coefficients of the differential equation~\eqref{eq:diffeq} and use the known $G_0^{q_j}$, $H_0^{q_j}$, $I_0^{q_j}$ values.
\begin{subequations}
\label{eq:log_vanishing_2}
\begin{align}
  \frac{G(z)}{\psi(z)} & = -\frac{1}{z-q_j} + G_1^{q_j} +\mathcal{O}(z-q_j), \\
  \frac{H(z)}{(\psi(z))^2} & = 0 + \frac{H_1^{q_j}}{z-q_j} + H_2^{q_j} +\mathcal{O}(z-q_j), \\
  \frac{I(z)}{(\psi(z))^3} & = 0 + \frac{I_1^{q_j}}{(z-q_j)^2} + \frac{I_2^{q_j}}{z-q_j} + I_3^{q_j} +\mathcal{O}(z-q_j).
\end{align}
\end{subequations}

Compare the coefficients from the series expansions of the coefficients of the Wronskian and the coefficients from the series expansions of coefficients of Equation~\eqref{eq:diffeq}, namely Equations~\eqref{eq:log_vanishing} and Equations~\eqref{eq:log_vanishing_2}. Note that the value $H_1^{q_j}$ depends on $b_1$ and $b_2$ only, hence $H_1^{q_j}$ will be the \emph{free parameter} (possible up to an affine transformation) which is denoted by $p_j$  \cite{dubrovin}. These will be the remaining $N$ Darboux coordinates on the moduli space next to the $q_j$'s. Elimination of Taylor coefficients from Conditions~\eqref{eq:log_vanishing} yield:
\begin{subequations}
\label{eq:ghi_123}
\begin{align}
  \label{eq:i1}
	I_1^{q_j} &= 0, \\
	\label{eq:h1}
  H_1^{q_j} &= p_j, \\
  \label{eq:h2i2}
  G_1^{q_j} H_1^{q_j} + \left( H_1^{q_j} \right)^2 + H_2^{q_j} + I_2^{q_j} &= 0, \\
  \label{eq:i2i3}
  G_1^{q_j} I_2^{q_j} + H_1^{q_j} I_2^{q_j} + I_3^{q_j} &= 0. 
\end{align}
\end{subequations}

\begin{prop}
Equations~\eqref{eq:ghi_123} are equivalent to the vanishing of the logarithmic terms.
\end{prop}  \qed

Compute terms $G_1^{q_j}$, $H_2^{q_j}$, $I_2^{q_j}$ and $I_3^{q_j}$ from the definition of the Laurent series. We neglect terms $G_1^{q_j}$ now, because it will turn out in Subsection~\ref{sec:gj_comput} these are zeros.
{\allowdisplaybreaks
\begin{align*}
    H_1^{q_j} &= \frac{\mathrm{d}}{\mathrm{d}z} \left. \left[ (z-q_j)^2 \frac{H(z)}{\psi^2} \right] \right|_{z=q_j} = \mu_j H'(q_j), \\
    H_2^{q_j} &= \frac{1}{2!} \frac{\mathrm{d^2}}{\mathrm{d}z^2} \left. \left[ (z-q_j)^2 \frac{H(z)}{\psi^2} \right] \right|_{z=q_j} = p_j \frac{\tilde{\mu_j}}{\mu_j} + \frac{\mu_j}{2} H''(q_j), \\
		I_1^{q_j} &= \frac{\mathrm{d}}{\mathrm{d}z} \left. \left[ (z-q_j)^3 \frac{I(z)}{\psi^3} \right] \right|_{z=q_j} = \nu_j I'(q_j), \\
    I_2^{q_j} &= \frac{1}{2!} \frac{\mathrm{d^2}}{\mathrm{d}z^2} \left. \left[ (z-q_j)^3 \frac{I(z)}{\psi^3} \right] \right|_{z=q_j} = \frac{\nu_j}{2} I''(q_j), \\
    I_3^{q_j} &= \frac{1}{3!} \frac{\mathrm{d^3}}{\mathrm{d}z^3} \left. \left[ (z-q_j)^3 \frac{I(z)}{\psi^3} \right] \right|_{z=q_j} = \frac{\tilde{\nu}_j}{2} I''(q_j) + \frac{\nu_j}{6} I'''(q_j),
\end{align*}}
where we introduce some notations:
\begin{subequations}
\label{eq:mu-nu}
\begin{align}
  \mu_j & := \left. \frac{(z-q_j)^2}{\psi^2 (z)} \right|_{z=q_j}, \\
  \tilde{\mu}_j &:= \left. \frac{\mathrm{d}}{\mathrm{d}z} \left[ \frac{(z-q_j)^2}{\psi^2 (z)}\right] \right|_{z=q_j}, \\
  \nu_j & := \left. \frac{(z-q_j)^3}{\psi^3 (z)}\right|_{z=q_j}, \\
  \tilde{\nu}_j & := \left. \frac{\mathrm{d}}{\mathrm{d}z} \left[ \frac{(z-q_j)^3}{\psi^3 (z)}\right] \right|_{z=q_j}.
\end{align}
\end{subequations}

Substitute these $H_1^{q_j}$, $H_2^{q_j}$, $I_1^{q_j}$, $I_2^{q_j}$ and $I_3^{q_j}$ terms to Equations~\eqref{eq:ghi_123} and rearrange:
\begin{subequations}
\begin{align}
  \label{eq:iqj_second}
  \sum_{k=1}^{12n-18} k q_j^{k-1} I_k &= 0, \\
  \label{eq:hqj_second}
  \sum_{k=1}^{8n-12} k q_j^{k-1} H_k &= \frac{p_j}{\mu_j}, \\
	\label{eq:hqj-iqj_0}
  \mu_j \frac{H''(q_j)}{2} + \nu_j \frac{I''(q_j)}{2} &= - G_1^{q_j} p_j - p_j^2 - \frac{\tilde{\mu}_j}{\mu_j} p_j, \\
	\label{eq:iqj_third_0}
  \left( (G_1^{q_j} + p_j) \nu_j + \tilde{\nu}_j \right) \frac{I''(q_j)}{2} + \nu_j \frac{I'''(q_j)}{6} &= 0,
\end{align}
\end{subequations}
where introduce the following notation:
\begin{equation}
		\label{eq:omega}
    \omega_j := (G_1^{q_j} + p_j) \nu_j + \tilde{\nu}_j. 
\end{equation}

Finally, from Equations~\eqref{eq:hqj-iqj_0} and \eqref{eq:iqj_third_0} we get $3n-5$ linear equations in the variables $H_k$ and $I_k$ ($k=2,\dots, 8n-12$ or $k=2,\dots, 12n-18$ respectively):
\begin{subequations}
\begin{align}
  \label{eq:hqj-iqj}
  \sum_{k=2}^{8n-12} \binom{k}{2} \mu_j q_j^{k-2} H_k + \sum_{k=2}^{12n-18} \binom{k}{2} \nu_j q_j^{k-2} I_k &= - G_1^{q_j} p_j - p_j^2 - \frac{\tilde{\mu}_j}{\mu_j} p_j, \\
  \label{eq:iqj_third}
  \sum_{k=2}^{12n-18} \left( \binom{k}{2} \omega_j q_j^{k-2} + \binom{k}{3} \nu_j q_j^{k-3} \right) I_k &= 0.
\end{align}
\end{subequations}

\begin{defn}
	The Equation~\eqref{eq:hqj-iqj} contains $3n-5$ relations for coefficients of $H$ and $I$ together. Later we will refer these $3n-5$ equations as \emph{common rows}.
\end{defn}

\subsection{Computation of \texorpdfstring{$G_1^{q_j}$}{G\textoneinferior(qj)}}
\label{sec:gj_comput}
The term $G_1^{q_j}$ appears in some previous equations as \eqref{eq:omega} and \eqref{eq:hqj-iqj}, thus we need to compute this value. 

Consider the right hand side of Equation~\eqref{eq:gqj} and introduce the notation
\begin{equation}
\label{eq:g_eta}
G(q_j) = - \frac{1}{\prod_{k=1}^n (q_j-t_k) \prod_{l=1, l\neq j}^{3n-5} (q_j-q_l)}=:-\eta_j
\end{equation} 
The term $G_1^{q_j}$ is the zeroth coefficient of the Laurent series of $\frac{G(z)}{\psi(z)}$ about $q_j$. Use Expression~\eqref{eq:psi} and expression~$\eta_j$:
\begin{equation*}
 G_1^{q_j} = \frac{\mathrm{d}}{\mathrm{d}z} \left. \left[ (z-q_j) \frac{G(z)}{\psi(z)} \right] \right|_{z=q_j}= \frac{\mathrm{d}}{\mathrm{d}q_j} \left( \eta_j G(q_j)\right).
\end{equation*}
Substitute Equation~\eqref{eq:g_eta} and conclude that we have to differentiate $-1$ which is zero, hence 
\begin{equation*}
	G_1^{q_j}=0.
\end{equation*}

\subsection{Summary of equations}
\label{sec:system_of_equations}
Equations~\eqref{eq:gt0}, \eqref{eq:gti} and \eqref{eq:gqj} give $4n-4$ relations between $4n-5$ coefficients of $G$, but the Fuchs relation (Equation~\eqref{eq:fuchs}) reduces the number of linearly independent equations by one. Equations~\eqref{eq:ht0}, \eqref{eq:hti}, \eqref{eq:hqj} and \eqref{eq:hqj_second} give $7n-9$ constraints for $H$, and Equations~\eqref{eq:it0}, \eqref{eq:iti}, \eqref{eq:iqj}, \eqref{eq:iqj_second} and \eqref{eq:iqj_third} give $10n-14$ constraints for $I$. Furthermore, Equation~\eqref{eq:hqj-iqj} contains $3n-5$ relations for coefficients of $H$ and $I$ together (these are the common rows). The system of linear equations consists Equations~\eqref{eq:ht0}, \eqref{eq:hti}, \eqref{eq:hqj}, \eqref{eq:hqj_second}, \eqref{eq:hqj-iqj}, \eqref{eq:it0}, \eqref{eq:iti}, \eqref{eq:iqj}, \eqref{eq:iqj_second} and \eqref{eq:iqj_third} will be denoted by $(T)$.

The total number of equations for the coefficients of $H$ and $I$ is $24n-33$. This is the same number as the number of coefficients contained in $H$ and $I$.

We note that the right hand side of this linear system of equations is known, since the parabolic points $t_i$ and their exponents $\rho_{i,k}$ are given; the apparent singularities $q_j$ and parameters $p_j$ are freely chosen. 

We need to prove that the linear system of equations in the variables of coefficients of $G$, $H$ and $I$ has a unique solution, i. e. the coefficient matrix has full rank. First, we study the cases $n=2$ and $n=3$ because the general case will use the method which will be written in the case $n=3$.

\begin{prop}
\label{prop:G_matrix}
	The coefficient matrix in question is a block diagonal matrix and the first block which refers to the variables of coefficients of $G$ has full rank.
\end{prop}
\begin{proof}
	The first $4n-4$ rows of the coefficient matrix are independent of other rows, because Equations~\eqref{eq:gt0}, \eqref{eq:gti} and \eqref{eq:gqj} do not contain coefficients of $H$ or $I$. Moreover, the Fuchs relation (Equation~\eqref{eq:fuchs}) reduces the rank by one, hence we may ignore the first row. The remaining block is a Vandermonde matrix which has nonzero determinant according to Proposition~\ref{prop:conf_vand}.
\end{proof}
According to this proposition the determinant of the block of coefficient matrix which determines the coefficients of $G$ does not vanish and it provides the following (trivial) discriminant variety in \mbox{$T^{*}(\mathbb{C} \setminus P)^4$}:
\begin{equation*}
	V_0:=\left\{ \prod_{\substack{1\leq i_1<i_2\leq 3 \\ 1\leq j_1<j_2\leq 4}} (t_{i_2}-t_{i_1}) \cdot (q_{j_2}-q_{j_1})
	 \prod_{\substack{1\leq i \leq 3 \\ 1\leq j\leq 4}} (q_{j}-t_{i}) =0 \right\}.
\end{equation*}
\noindent The variety $V_0$ coincides with the subset $\Delta$ hence we do not deal with the block of coefficient matrix which determines the coefficients of $G$.

\section{The case \texorpdfstring{$n=3$}{n=3}}
\label{sec:n=3}
In this section we will give an argument supporting for Conjecture~\ref{cj:main_V2}. Consider System~$(T)$ from Subsection~\ref{sec:system_of_equations} and its coefficient matrix $M_1$. The determinant of this matrix will define a discriminant variety $V_1$ on \mbox{$T^{*}(\mathbb{C} \setminus P)^4$}. 
The augmented matrix of System $(T)$ will be $M_b$ which will define another discriminant variety $\hat{V}$.
Our purpose is to analyze the special case when $\mathsf{rk}(M_1)=\mathsf{rk}(M_b) < 20n-28 = 32$ that is the System~$(T)$ is under-determined. We will analyze the case $\mathsf{rk}(M_1)=31$ only. For this purpose, we pick a minor of $M_1$ which defines a variety $W$. Finally, we will specify a one-parameter family of solutions of System~$(T)$ in the subvariety $(V_1 \cap \hat{V}) \setminus W$. 

\subsection{The case \texorpdfstring{$n=2$}{n=2}}
Before we analyze the case $n=3$, for the sake of completeness, we review the case $n=2$ shortly. 
The parabolic points: $P= \left\{ t_0, t_1, t_2 \right\}$, the apparent singularity: $Q= \left\{ q_1 \right\}$, the coefficients of Equation~\eqref{eq:diffeq}:
\begin{align*}
  G(z) & = G_0 + G_1 z + G_2 z^2, \\
  H(z) & = H_1 + H_1 z + H_2 z^2 + H_3 z^3 + H_4 z^4, \\
  I(z) & = I_0 + I_1 z + I_2 z^2 + I_3 z^3 + I_4 z^4 + I_5 z^5 + I_6 z^6.
\end{align*}
The constants, which appear in the coefficient matrix are the following:
\begin{align*}
  \mu_1 & = \frac{1}{\left(q_1-t_1\right){}^2 \left(q_1-t_2\right){}^2}, \\
  \nu_1 & = \frac{1}{\left(q_1-t_1\right){}^3 \left(q_1-t_2\right){}^3}, \\
  \omega_1 & = -\frac{3}{\left(q_1-t_1\right){}^3 \left(q_1-t_2\right){}^4}-\frac{3}{\left(q_1-t_1\right){}^4 \left(q_1-t_2\right){}^3}+\frac{p_1+g_1}{\left(q_1-t_1\right){}^3 \left(q_1-t_2\right){}^3}.
\end{align*}

Use System~$(T)$ and get a coefficient matrix for the variables of coefficients of $H$ and $I$.

\begin{scriptsize}
\begin{equation*}
\left( \begin{array}{cccccccccccc}
 0 & 0 & 0 & 0 & 1 & 0 & 0 & 0 & 0 & 0 & 0 & 0 \\
 1 & t_1 & t_1{}^2 & t_1{}^3 & t_1{}^4 & 0 & 0 & 0 & 0 & 0 & 0 & 0 \\
 1 & t_2 & t_2{}^2 & t_2{}^3 & t_2{}^4 & 0 & 0 & 0 & 0 & 0 & 0 & 0 \\
 1 & q_1 & q_1{}^2 & q_1{}^3 & q_1{}^4 & 0 & 0 & 0 & 0 & 0 & 0 & 0 \\
 0 & 1 & 2 q_1 & 3 q_1{}^2 & 4 q_1{}^3 & 0 & 0 & 0 & 0 & 0 & 0 & 0 \\
 0 & 0 & \mu_1 & 3\mu_1 q_1 & 6\mu_1 q_1{}^2 & 0 & 0 & \nu _1 & 3\nu _1 q_1 & 6\nu _1 q_1{}^2 & 10\nu _1 q_1{}^3 & 15\nu
_1 q_1{}^4 \\
 0 & 0 & 0 & 0 & 0 & 0 & 0 & 0 & 0 & 0 & 0 & 1 \\
 0 & 0 & 0 & 0 & 0 & 1 & t_1 & t_1{}^2 & t_1{}^3 & t_1{}^4 & t_1{}^5 & t_1{}^6 \\
 0 & 0 & 0 & 0 & 0 & 1 & t_2 & t_2{}^2 & t_2{}^3 & t_2{}^4 & t_2{}^5 & t_2{}^6 \\
 0 & 0 & 0 & 0 & 0 & 1 & q_1 & q_1{}^2 & q_1{}^3 & q_1{}^4 & q_1{}^5 & q_1{}^6 \\
 0 & 0 & 0 & 0 & 0 & 0 & 1 & 2 q_1 & 3 q_1{}^2 & 4 q_1{}^3 & 5 q_1{}^4 & 6 q_1{}^5 \\
 0 & 0 & 0 & 0 & 0 & 0 & 0 & \omega _1 & 3\omega _1q_1+\nu _1 & \Lambda_1 & \Lambda_2 & \Lambda_3
\end{array} \right),
\end{equation*}
\end{scriptsize}

\noindent where $\Lambda_1 = 6\omega _1q_1{}^2+4\nu _1 q_1$, $\Lambda_2 = 10\omega _1q_1{}^3+10\nu _1 q_1{}^2$, $\Lambda_3 = 15\omega _1q_1{}^4+20\nu$. 
Compute its determinant with Laplace expansion along the sixth row. The determinant is $-(t_1-t_2)^2$ which does not vanish since $t_i \neq t_j$ if $i \neq j$. Thus the coefficients of $G$, $H$ and $I$ are determined, the differential equation~\eqref{eq:diffeq} is also determined and Theorem~\ref{th:main_V1} is valid in the case $n=2$. The set $V$ is the empty set.

\subsection{Computation of a discriminant variety}
\label{sec:M1}
In this subsection we construct the variety $V_1$. From now on, let $n=3$ and fix the parabolic points and apparent singularities as $P=\left\{ t_0, t_1, t_2, t_3\right\}$ and $Q=\left\{ q_1, q_2, q_3, q_4 \right\}$. The free parameters are $\overset{\sim}{Q} = \left\{p_1, p_2, p_3, p_4 \right\}$. The numerators of coefficients of differential equation~\eqref{eq:diffeq} are:
\begin{align*}
  G(z) &= G_0 + G_1 z + \dots+ G_{6} z^{6}, \\
  H(z) &= H_0 + H_1 z + \dots+ H_{12} z^{12}, \\
  I(z) &= I_0 + I_1 z + \dots+ I_{18} z^{18}, \\
	\intertext{and their denominators are successive powers of}
  \psi(z) &= \prod_{i=1}^3 (z-t_i) \prod_{j=1}^{4} (z-q_j).
\end{align*}
The expression~$\psi (z)$, the sets $P$, $Q$ and $\overset{\sim}{Q}$ and the exponents~$\rho_{i,k}$ ($i=0,\dots, 3$; $k=1,2,3$) specify the constants $\mu_j$, $\nu_j$, $\tilde{\nu}_j$ and $\omega_j$. The index $j$ refers to $1,\dots, 4$ through this section.
Again, consider System~$(T)$ from Subsection~\ref{sec:system_of_equations} and get the following $32 \times 32$ matrix (denoted by $M_1$) for the variables of coefficients of $H$ and $I$. 
\begin{scriptsize}
\begin{displaymath}
\left(
\begin{array}{ccccccccccccc}
 0 & 0 & 0 & 0 & \cdots & 1 & 0 & 0 & 0 & 0 & \cdots & 0 \\
 1 & t_1 & t_1^2 & t_1^3 & \cdots & t_1^{12} & 0 & 0 & 0 & 0 & \cdots & 0 \\
 1 & t_2 & t_2^2 & t_2^3 & \cdots & t_2^{12} & 0 & 0 & 0 & 0 & \cdots & 0 \\
 1 & t_3 & t_3^2 & t_3^3 & \cdots & t_3^{12} & 0 & 0 & 0 & 0 & \cdots & 0 \\
 1 & q_1 & q_1^2 & q_1^3 & \cdots & q_1^{12} & 0 & 0 & 0 & 0 & \cdots & 0 \\
 1 & q_2 & q_2^2 & q_2^3 & \cdots & q_2^{12} & 0 & 0 & 0 & 0 & \cdots & 0 \\
 1 & q_3 & q_3^2 & q_3^3 & \cdots & q_3^{12} & 0 & 0 & 0 & 0 & \cdots & 0 \\
 1 & q_4 & q_4^2 & q_4^3 & \cdots & q_4^{12} & 0 & 0 & 0 & 0 & \cdots & 0 \\
 0 & 1 & 2 q_1 & 3 q_1^2 & \cdots & 12 q_1^{11} & 0 & 0 & 0 & 0 & \cdots & 0 \\
 \vdots &  &  &  &  &  &  &  &  &  &  & \\
 0 & 0 & \mu _1 & 3 q_1 \mu _1 & \cdots & 66 q_1^{10} \mu _1 & 0 & 0 & \nu _1 & 3 q_1 \nu _1 & \cdots & 153 q_1^{16} \nu _1 \\
 0 & 0 & \mu _2 & 3 q_2 \mu _2 & \cdots & 66 q_2^{10} \mu _2 & 0 & 0 & \nu _2 & 3 q_2 \nu _2 & \cdots & 153 q_2^{16} \nu _2 \\
 0 & 0 & \mu _3 & 3 q_3 \mu _3 & \cdots & 66 q_3^{10} \mu _3 & 0 & 0 & \nu _3 & 3 q_3 \nu _3 & \cdots & 153 q_3^{16} \nu _3 \\
 0 & 0 & \mu _4 & 3 q_4 \mu _4 & \cdots & 66 q_4^{10} \mu _4 & 0 & 0 & \nu _4 & 3 q_4 \nu _4 & \cdots & 153 q_4^{16} \nu _4 \\
 0 & 0 & 0 & 0 & \cdots & 0 & 0 & 0 & 0 & 0 & \cdots & 1 \\
 0 & 0 & 0 & 0 & \cdots & 0 & 1 & t_1 & t_1^2 & t_1^3 & \cdots & t_1^{18} \\
\vdots &  &  &  &  &  &  &  &  &  &  &  \\
 0 & 0 & 0 & 0 & \cdots & 0 & 1 & q_1 & q_1^2 & q_1^3 & \cdots & q_1^{18} \\
\vdots &  &  &  &  &  &  &  &  &  &  &  \\
 0 & 0 & 0 & 0 & \cdots & 0 & 0 & 1 & 2 q_1 & 3 q_1^2 & \cdots & 18 q_1^{17} \\
\vdots &  &  &  &  &  &  &  &  &  &  &  \\
 0 & 0 & 0 & 0 & \cdots & 0 & 0 & 0 & \omega _1 & \nu _1+3 q_1 \omega _1 & \cdots & 816 q_1^{15} \nu _1+153 q_1^{16} \omega _1 \\
\vdots &  &  &  &  &  &  &  &  &  &  &  \\
\end{array}
\right).
\end{displaymath}
\end{scriptsize}

\subsubsection{Searching for confluent Vandermonde matrices}
\label{sec:conf_Vand_search}
The determinant of $M_1$ cannot be computed directly even with computers. $M_1$ is almost a block diagonal matrix with two blocks. It is just almost block diagonal because the Equation~\eqref{eq:hqj-iqj} provides for the matrix four rows (the so-called common rows) which have nonzero element in both blocks. These quasi blocks are \emph{almost} confluent Vandermonde matrices in the sense that they have several rows which together form confluent Vandermonde matrices.

Our aim is to transform $M_1$ to some block lower/upper triangular matrices where the upper $13 \times 13$ and lower $19 \times 19$ blocks are confluent Vandermonde matrices and the remaining $13 \times 19$ or $19 \times 13$ blocks are zero matrices.

In the first step divide $M_1$ into two other matrices where the key object is the 13th row:
\begin{equation*}
	\begin{array}{ccccccccccccc}
	0 & 0 & \mu _1 & 3 q_1 \mu _1 & \cdots & 66 q_1^{10} \mu _1 & 0 & 0 & \nu _1 & 3 q_1 \nu _1 & \cdots & 153 q_1^{16} \nu _1.
	\end{array}
\end{equation*}
The first matrix (denoted by $M_1^1$) only differs from $M_1$ in that the 14th--32nd elements of the 13th row are zeros. The second matrix (say $M_1^2$) is the so-called auxiliary matrix, $M_1^2$ differs from $M_1$ in that the 1st--13th elements of the 13th row are zeros. Obviously $\det (M_1) = \det (M_1^1) + \det (M_1^2)$ and the matrix $M_1^1$ is a block lower triangular with the desired shape, since it has a $19 \times 13$ zero block. (Note, the proof that $M_1^1$ is confluent Vandermonde still remains.) $R_1$ and $S_1$ denote the diagonal blocks of $M_1^1$. 

In the second step interchange the 13th and the 14th rows in $M_1^2$. The determinant will be multiplied by $-1$. The resulting matrix (denoted by ${M_1^2}'$) has the shape of $M_1$ except some zero elements in the 14th row. (Note that the 13th row contains powers of $q_2$ and the 14th row contains powers of $q_1$.) Apply the first step to ${M_1^2}'$ and get two new matrices. One of them is a block lower triangular matrix with diagonal blocks $R_2$ and $S_2$. 

Continue the process (second step and first step) until the $13 \times 19$ or the $19 \times 13$ blocks are zero matrices. We need two more iterations and denote the diagonal blocks of the created matrices by $R_3$, $S_3$, $R_4$ and $S_4$. Finally, we get another matrix where the 1st--13th elements of 13th--32nd rows are zeros, thus the determinant of this matrix vanishes.

Hence the determinant of $M_1$ is the following:
\begin{equation}
\label{eq:detM1}
  \det(M_1) = \sum_{j=1}^4 \det(R_j) (-1)^{j-1} \det (S_j),
\end{equation}
where the $j$ also refers to $q_j$ which is contained in the 13th row in $R_j$ and to the $q_j$ which is not contained in the first three rows in $S_j$.

The last row of $R_j$ has a factor $\mu_j$ ($j=1,\dots,4$). Denote the matrix without this factor by $R'_j$ and notice that $R'_j$ is confluent Vandermonde apart from its first row.

The determinant of $S_j$ ($j=1,\dots, 4$) does not change if we add a scalar multiple of a row to another row. Based on Equations~\eqref{eq:hqj-iqj} and \eqref{eq:iqj_third} we can change the matrix $S_j$:
\begin{displaymath} 
  S_j[\tilde{k},.] := S_j[\tilde{k},.] - \frac{\omega_k}{\nu_k} S_j[k,.], \hspace{10mm} (k=1,\dots, 3), 
\end{displaymath}
where we abuse the notation and also denote the new matrix by $S_j$. The term $[k,.]$ denotes the $k$-th row of $S_j$. 
The index $\tilde{k} \in \{16,\dots, 19\}$ denotes the row which contains the same $q_l$ powers as the $k$-th row. For example, in $S_3$ we leave $k=1, 2, 3$ and $\tilde{k}=16,17,19$ respectively.

Divide $S_j$ to two other matrices along the $(15+j)$-th row. The first matrix contains the second derivative of $I(q_j)$, and if we factor $\omega_j \nu_1^2 \nu_2^2 \nu_3^2 \nu_4^2 / \nu_j^2$ from the rows, then the remaining matrix will be a confluent Vandermonde and will be denoted by $S'_j$. The second matrix contains the third derivative of $I(q_j)$. Factor $\nu_1^2 \nu_2^2 \nu_3^2 \nu_4^2 / \nu_j$ and denote the remaining matrix by $\hat{S}'_j$. The matrix $\hat{S}'_j$ is an almost Vandermonde, it misses the second derivative of $I(q_j)$. However, $\det (\hat{S}'_j)$ can be derived from $\det (S'_j)$ as follows:
\begin{equation}
  \label{eq:s_hat_det}
  \hat{s}_j := \det (\hat{S}'_j) = \frac{1}{3} \frac{\mathrm{d}}{\mathrm{d}q_j} \det (S'_j).
\end{equation}
Hence we showed that $M_1$ can be decomposed as some block lower triangular matrices with two confluent Vandermonde matrices in its diagonals.

\subsubsection{Matrices and their determinants}
\label{sec:M1_matr_det}
The parameters in the rows of $R'_1$ are $1$, $t_1$, $t_2$, $t_3$, $q_1^{(0)}$, $q_2^{(0)}$, $q_3^{(0)}$, $q_4^{(0)}$, $q_1^{(1)}$, $q_2^{(1)}$, $q_3^{(1)}$, $q_4^{(1)}$, $q_1^{(2)}$ respectively, where the upper index refers to the order of derivative. This sequence differs from the \emph{standard sequence} (see Definition~\ref{def:st_seq}).

\begin{defn}
\label{def:inversion}
	We introduce a \emph{total order} ($\prec$) on the set of $t_i$'s and $q_j^{(k)}$'s. We define the order such that the ascending order is just the standard sequence. 
	Let $L$ be a sequence with $t_i$'s and $q_j^{(k)}$'s and we label its elements by $\{1,2,3,\dots\}$. If $\alpha<\beta$ and $L(\alpha)\succ L(\beta)$ is called an \emph{inversion} of $L$.
\end{defn}
The number of inversions equals the number of necessary row interchanges which converts the sequence of parameters to the standard sequence. The row interchanges change the determinant by a sign hence it is necessary to calculate the parity of number of inversions.

The number of inversions is $12$ in the matrix $R'_1$, thus apply Equation~\eqref{eq:conf_vand} and get the determinant of $R'_1$:
\begin{equation}
\label{eq:r1}
  \begin{split}
    r_1:=&\det (R'_1) = \left(-q_1+q_2\right){}^6 \left(-q_1+q_3\right){}^6 \left(-q_2+q_3\right){}^4 \left(-q_1+q_4\right){}^6 \left(-q_2+q_4\right){}^4 \cdot \\
    & \cdot \left(-q_3+q_4\right){}^4 \left(q_1-t_1\right){}^3 \left(q_2-t_1\right){}^2 \left(q_3-t_1\right){}^2 \left(q_4-t_1\right){}^2 \left(q_1-t_2\right){}^3 \left(q_2-t_2\right){}^2 \cdot \\
    & \cdot \left(q_3-t_2\right){}^2 \left(q_4-t_2\right){}^2\left(-t_1+t_2\right) \left(q_1-t_3\right){}^3 \left(q_2-t_3\right){}^2 \left(q_3-t_3\right){}^2 \cdot \\ 
		& \cdot \left(q_4-t_3\right){}^2 \left(-t_1+t_3\right) \left(-t_2+t_3\right).
  \end{split}
\end{equation}

The determinants of the other $R'_j$'s (denoted by $r_j$, $j=2, 3, 4$) can be computed similarly but there is a simpler way if we change the indices. Change $q_j$ to $q_1$ and vice versa in Equation~\eqref{eq:r1}. We only need to compute how many row interchanges are needed to the conversion between $R'_j$ to $R'_1$. The numbers of inversions are even in all $R'_j$'s, thus the determinants $r_2$, $r_3$ and $r_4$ can be calculated with index change from $r_1$ only. We do not write these terms due to limited space. 

The parameters in $S'_1$ are: $q_2^{(2)}$, $q_3^{(2)}$, $q_4^{(2)}$, $1$, $t_1$, $t_2$, $t_3$, $q_1^{(0)}$, $q_2^{(0)}$, $q_3^{(0)}$, $q_4^{(0)}$, $q_1^{(1)}$, $q_2^{(1)}$, $q_3^{(1)}$, $q_4^{(1)}$, $q_1^{(2)}$, $q_2^{(3)}$, $q_3^{(3)}$, $q_4^{(3)}$. The number of inversions is $44$. Use Equation~\eqref{eq:conf_vand} again for the determinant of $S'_1$:
\begin{equation}
\label{eq:s1}
  \begin{split}
   s_1 =& \left(-q_1+q_2\right){}^{12} \left(-q_1+q_3\right){}^{12} \left(-q_2+q_3\right){}^{16} \left(-q_1+q_4\right){}^{12} \left(-q_2+q_4\right){}^{16} \cdot\\
       & \cdot \left(-q_3+q_4\right){}^{16} \left(q_1-t_1\right){}^3\left(q_2-t_1\right){}^4 \left(q_3-t_1\right){}^4 \left(q_4-t_1\right){}^4 \left(q_1-t_2\right){}^3 \left(q_2-t_2\right){}^4 \cdot\\
       & \cdot \left(q_3-t_2\right){}^4 \left(q_4-t_2\right){}^4\left(-t_1+t_2\right) \left(q_1-t_3\right){}^3 \left(q_2-t_3\right){}^4 \left(q_3-t_3\right){}^4 \left(q_4-t_3\right){}^4 \cdot \\
			& \cdot \left(-t_1+t_3\right) \left(-t_2+t_3\right).
  \end{split}
\end{equation}

For the determinants of other $S'_j$'s ($j=2, 3, 4$) use the similar method as for $R'_j$'s. Namely, interchange parameters $q_j$ and $q_1$. It is easy to compute the parity of the number of inversions: $(-1)^{3+(j-2)}$. Note that the sign $(-1)^{j-1}$ in Equation~\eqref{eq:detM1} cancels the previous sign, hence we can ignore these signs in the determinant of $M_1$. Get $s_j:=\det S'_j$ that is, interchange $q_1$ and $q_j$ in Equation~\eqref{eq:s1}. We do not write these terms due to limited space. 

Next, produce $\hat{s}_j$ with Equation~\eqref{eq:s_hat_det}.

Finally, sum up all $r_j$'s, $s_j$'s and $\hat{s}_j$'s, substitute $\mu_j$, $\nu_j$ and $\omega_j$ and get the determinant of $M_1$. 
\begin{equation*}
  \sigma_1  := \det (M_1) = \sum_{j=1}^4 \mu_j r_j \left(\frac{\omega_j \nu_1^2 \nu_2^2 \nu_3^2 \nu_4^2}{\nu_j^2} s_j + \frac{\nu_1^2 \nu_2^2 \nu_3^2 \nu_4^2}{\nu_j} \hat{s}_j \right).
\end{equation*}

After long simplification a symmetric polynomial $\chi_1$ can be factored out (namely $\sigma_1 = \chi_1 \phi_1$), where
\begin{equation*}
  \chi_1 = \prod_{1 \leq j_1 < j_2 \leq 4} (q_{j_1}-q_{j_2})^6 \prod_{1 \leq i_1 < i_2 \leq 3} (t_{i_1}-t_{i_2})^2.
\end{equation*}
The polynomial $\phi_1$ is also symmetric for permutations from $\Sym(4)$:
\begin{equation*}
  \begin{split}
    \phi_1 =& \sum_{j=1}^4 \bigg[ \prod_{\substack{ k_1, k_2 = 1 \\ k_1, k_2 \neq j,\; k_1 < k_2 }}^4 (q_{k_1} - q_{k_2})^2 \prod_{\substack{l = 1 \\ l \neq j}}^4 (q_j-q_l) \prod_{m=1}^3 (q_j-t_m) \bigg] \cdot \\
	 &\cdot \left[ p_j+\sum_{l=1,\, l \neq j}^4 \frac{1}{q_j-q_l}-2\sum_{m=1}^3 \frac{1}{q_j-t_m} \right],
  \end{split}
\end{equation*}
where the denominators of the fractions cancel with multiplier factors, hence $\phi_1$ is a polynomial.

The matrix $M_1$ has full rank if and only if its determinant does not vanish. The polynomial $\chi_1$ is not zero because $q_j \neq t_i$for any $i$ and $j$. In order to show that the polynomial $\phi_1$ does not vanish, it is enough to compute the coefficient of highest power in $q_1$. The highest power is $q_1^6$ and its coefficient is
\begin{displaymath} 
  p_1 \left(q_2-q_3\right)^2 \left(q_2-q_4\right)^2 \left(q_3-q_4\right)^2,
\end{displaymath}
what is nonzero if the $q_j$'s differ from each other. 

Hence we proved that $M_1$ has full rank in the generic case, moreover $\sigma_1$ specifies a discriminant variety in $T^{*}(\mathbb{C} \setminus P)^4$ (denoted by $V_1$). Consequently, the moduli space $\Mod$ of logarithmic connections of a rank-$3$ vector bundle over the Riemann sphere with $3$ parabolic points can be equipped with Darboux coordinates $q_j, p_j$ outside of $V_1$. This proves Theorem~\ref{th:main_V1} in case $n=3$. 

\subsection{Further discriminant varieties}
\label{sec:M2}
It is well known from the theory of system of linear equations that if the coefficient matrix does not have full rank but this rank equals the rank of augmented matrix then the system of equations is under-determined and the solution has some free parameters.
We want to reproduce this phenomenon related to the moduli space $\Mod$.

In this subsection we create varieties (denoted by $V_k$, $k=2,\dots,33$) which corresponds to the augmented matrix and we show a numerical example where $V_1\neq V_k$ and $V_1 \cap V_k\neq \emptyset$ for some index $k$.

Denote the right-hand vector of the system of linear equations~$(T)$ by $\mathbf{b}$:
\begin{footnotesize}
\begin{equation*}
\mathbf{b}^T = \left( \beta_0-\alpha_0+1, \frac{\beta_i-\alpha_i+1}{\lambda_i^2}, 0, 0, 0, 0, \frac{p_j}{\mu_j}, p_j^2 - \frac{\tilde{\mu}_j}{\mu_j} p_j, -\gamma_0, -\frac{\gamma_i}{\lambda_i^3}, 0, 0, 0, 0, 0, 0, 0, 0, 0, 0, 0, 0 \right)\!,
\end{equation*}
\end{footnotesize}
where $i=1, 2, 3$; $j=1, \dots, 4$, and introduce some notations:
\begin{alignat*}{3}
\alpha_i &= \rho_{i,1} + \rho_{i,2} + \rho_{i,3}, &\mbox{~where~} i=0, \dots, 3, \\
\beta_i &= \rho_{i,1} \rho_{i,2} + \rho_{i,1} \rho_{i,3} + \rho_{i,2} \rho_{i,3}, &\mbox{~where~} i=0, \dots, 3, \\
\gamma_i &= \rho_{i,1} \rho_{i,2} \rho_{i,3}, &\mbox{~where~} i=0, \dots, 3, \\
\lambda_i &=\frac{1}{\prod_{k=1,\; k\neq i}^3 (t_i-t_k) \prod_{l=1}^4 (t_i-q_l)}, &\mbox{~where~} i=1, \dots, 3,\\
\mu_j &=\frac{1}{\prod_{k=1}^3 (q_j-t_k)^2 \prod_{l=1,\; l\neq j}^4 (q_j-q_l)^2}, &\mbox{~where~} j=1, \dots, 4,\\
\tilde{\mu}_j &=\frac{\mathrm{d}\mu_j}{\mathrm{d}q_j}, &\mbox{~where~} j=1, \dots, 4.
\end{alignat*}
For the sake of completeness we express explicitly the parameters $\nu_j$, $\tilde{\nu}_j$ and $\omega_j$ which are introduced in Subsection~\ref{sec:app_sing} ($j=1, \dots, 4$).
\begin{align*}
	\nu_j & = \frac{1}{\prod_{k=1}^3 (q_j-t_k)^3 \prod_{l=1,\; l\neq j}^4 (q_j-q_l)^3} = \mu_j^{\frac{3}{2}}, \\
  \tilde{\nu}_j &= \frac{\mathrm{d}\nu_j}{\mathrm{d}q_j}, \\
	\omega_j & = p_j \nu_j + \tilde{\nu}_j. 
\end{align*}

Denote by $M_b$ that $32 \times 33$ matrix which has first column $\mathbf{b}$ and the others are the same as $M_1$. Denote by $M_k$ the $32 \times 32$ matrix which comes from $M_b$ by deleting its $k$-th column ($k=2,\dots,33$). Let $\sigma_k:=\det M_k$ and let $V_k$ be the variety in $T^{*}(\mathbb{C} \setminus P)^4$ which is defined by $\sigma_k$. (Note that the above definition gives back $M_1$ and $\sigma_1$ for $k=1$.)

It is clear that if one of $M_k$ ($k\geq 1$) has full rank, then the augmented matrix of System~$(T)$ has full rank. We proceed by method of contraposition. Namely, if the augmented matrix does not have full rank then $M_k$ for all $1\leq k \leq 33$ does not have full rank. Hence the subvariety of $\hat{V}:=\bigcap_{k=2}^{33} V_k$ describes the locus where the the moduli space does not admit unique coordinates. 
The moduli space $\Mod$ can be equipped coordinates on $V_k \setminus \hat{V}$, but it remains to understand what is happening in the subvariety $V_1 \cap \hat{V}$.

\subsubsection{Numerical example}
We do not compute varieties $V_k$ directly due to huge matrices and rational fractions. 
Thus we make a general numerical example which does not prove the Conjecture~\ref{cj:main_V2} but supports it. 
We choose random rational numbers for the parameters ($t_i, \alpha_i,\beta_i,\gamma_i$) and coordinates ($q_j, p_j$). Technically we choose rationals from the interval $[1/40,70]$ in order for the numbers in calculations be manageable easier. 

We can convince ourselves about $V_1 \neq V_k$ by substitution of random numbers to $\sigma_1$ and $\sigma_k$. 

\begin{itemize}
\item If $k=14$ or $k=33$ then $\sigma_{14}$ or $\sigma_{33}$ differ from $\sigma_1$ only a factor $(\beta_0 -\alpha_0 -1)$ or $(-\gamma_0)$. 
Thus $V_1=V_{14}=V_{33}$.
\item If $k=2$ or $k=15$ we projective transform the system. We already eliminated the parameter $t_0$ in the beginning of Section~\ref{sec:third_order}. Now, we can choose $t_1=0$, thus the $2$nd row in matrix $M_1$ becomes $(1, 0,\dots,0)$ and the $18$th row becomes $(0,\dots,0,1,0\dots,0)$ where the $1$ stands in the $14$th place.
There is a similar case than $k=14$ and $k=33$: the projective transformed $\sigma_2$ and $\sigma_{15}$ differ from $\sigma_1$ only by a factor 
$(\beta_1 -\alpha_1 -1)$ or $(-\gamma_1)$. Hence $V_1=V_{2}=V_{15}$.
\item In the remaining cases ($k\neq 2,14,15,33$) it turns out that $V_1 \neq V_k$.
\end{itemize}
\noindent In the rest of this section the index $k \in \left\{2,\dots,33\right\} \setminus \left\{2,14,15,33\right\}$.

First, we will find a point $P$ in $V_1 \cap \hat{V}$. Fix all parameters in $\sigma_1$ and $\sigma_k$ except for two of them. 
For the sake of simplicity let the two non-fixed parameters be $p_1$ and $p_2$ since $\deg_{p_j} \sigma_1 = 1$ and $\deg_{p_j} \sigma_k = 2$ for $j=1,\dots,4$. Let the fixed parameters be random rational numbers such that the factors $(q_{i_1}-t_{i_2})$, $(q_j-t_i)$ and $(q_{j_1}-q_{j_2})$ do not vanish. We note $\beta_j=0$ is a possible choice because $\alpha_j$ and $\beta_j$ appear only in the factor $(\alpha_j-\beta_j-1)$.

Substitute these numbers to $\sigma_1$ and $\sigma_k$, and solve the equation $\sigma_1=0$ for the variable $p_1$. Substitute the solution to the equation $\sigma_k=0$ and solve the resulting quadric equation for the variable $p_2$. 

\noindent The random numbers specify some points in $V_1 \cap V_k$ for all $k$. It turns out there exist three points in $V_1 \cap \hat{V}$ under the parameters fixed above. One is chosen as $P$. Geometrically, this means that the $p_1-p_2$ coordinate plane intersects $V_1 \cap \hat{V}$.
In the conjecture we suppose that the point $P$ is a smooth point in the varieties $V_1$ and $V_k$.

\subsection{Computation of a minor}
\label{sec:minor}
In this subsection we compute an open subset of $V_1 \cap \hat{V}$ with codimension $2$ (denoted by $(V_1 \cap \hat{V})^0$) and we want to specify the value of corank of $M_1$ on $(V_1 \cap \hat{V})^0$. If the rank decreases by $1$, then there exists a solution of the system of linear equations of coefficients of Equation~\eqref{eq:diffeq} on the subset $(V_1 \cap \hat{V})^0$. Examine the corank with a variety $W$ which defined by determinant of a certain minor of $M_1$. If there is a minor which does not vanish on an open subset of $(V_1 \cap \hat{V})$, then the rank of the system decreases by $1$ on the subset $(V_1 \cap \hat{V})\setminus W$. Moreover, this means: $\codim \left( V_1 \cap \hat{V} \cap W \right) \geq 3$.

We have the freedom to choose a minor, the appropriate choice is an easy computable one. If all varieties, defined by these minors, contain an everywhere dense subset of $V_1 \cap \hat{V}$, then the rank of $M_1$ decreases by at least $2$ in $V_1 \cap \hat{V}$.

Choose the minor associated to $M_1 [13, 13]$ and denote it by $M_f$. Compute $\det M_f$ with the method used so far. The minor has two diagonal blocks, denote the upper left by $R_f$ and the lower right by $S_f$. The upper right block is the zero matrix thus $\sigma_f := \det M_f=\det R_f \cdot \det S_f$.
\begin{equation*}
	\det R_f=\mathrm{ConfVand}\left(t_1^{(0)},t_2^{(0)},t_3^{(0)},q_1^{(1)},q_2^{(1)},q_3^{(1)},q_4^{(1)}\right).
\end{equation*}
The associated lower right block is $S_1$, hence:
\begin{equation*}
\det S_f = s_1 \omega_1 \nu_2^2 \nu_3^2 \nu_4^2 + \hat{s}_1 \nu_1 \nu_2^2 \nu_3^2 \nu_4^2.
\end{equation*}

The minor $\sigma_f$ defines a variety $W$ with $\codim W =1$.
The polynomial $\sigma_f$ factors as $\chi_f\phi_f$ and the two polynomials are:
\begin{align*}
\chi_f =& \prod_{1 \leq i_1 < i_2 \leq 3} (t_{i_1}-t_{i_2})^2 \prod_{i=1}^3 (t_i-q_1)^6  \prod_{j=2}^4 (q_1-q_j)^6 \prod_{\substack{k, l=2 \\ k<l}}^4 (q_{k}-q_{l})^8, \\
\begin{split}
\phi_f =& 2 \prod_{j=2}^4 (q_1-q_j) \left[ \sum_{\substack{k, l=1 \\ k<l}}^3 (q_1-t_{k})(q_1-t_{l}) \right] - \\
 &- \prod_{i=1}^3 (q_1-t_i) \left[ \sum_{\substack{k, l=2 \\ k<l}}^4 (q_1-q_k)(q_1-q_l) + p_1 (q_1-q_2) (q_1-q_3) (q_1-q_4) \right].
\end{split}
\end{align*}

The coefficient of the highest power in $q_1$ of $\phi_f$ is $-p_1$, hence $\phi_f \neq 0$. Moreover, $\sigma_f$ is generically irreducible.

Finally, substitute the computed point $P$ from Subsection~\ref{sec:M2} to $\sigma_f$. We get a nonzero number, hence the rank of $M_1$ decreases by $1$ at the point $P$. This is also true in an open neighborhood of $P$ due to conjecture that the varieties are smooth at $P$. Hence there exists an open subset of $V_1 \cap \hat{V}$, where the rank of $ M_1$ decreases by exactly $1$:
\begin{equation*}
(V_1 \cap \hat{V})^0 := (V_1 \cap \hat{V}) \setminus W \setminus \Delta.
\end{equation*}

\subsection{Blowing up along \texorpdfstring{$(V_1 \cap \hat{V})^0$}{(V\textoneinferior\ intersection V)\textzerosuperior}}
The (irreducible) varieties $V_1$ and $\hat{V}$ specify a $1$-dimensional linear system $D_{\lambda,k}=(\sigma_1 + \lambda \sigma_k =0)$ in $\Mod$ which defines a base locus 
\begin{equation*}
	\bigcap_{\substack{\lambda \in \mathbb{CP}^1 \\ k=2}}^33 D_{\lambda,k}. 
\end{equation*}
It is just the variety $V_1 \cap \hat{V}$.

We want to blow up along the base locus, but $D_{\lambda,k}$ is a pencil, hence we can fix only one parameter in the system of the linear equations $(T)$. Consequently, we do not get a unique solution where the rank of the system decreases by more than $1$.

The variety $V_1 \cap \hat{V} \cap W \setminus \Delta$ defines a multi-dimensional linear system but we present the blow up procedure along a $1$-dimensional linear system only, hence we blow up along $(V_1 \cap \hat{V})^0$.

Let $\Omega \subset \Mod$ be a $2N=8$ dimensional ball with holomorphic coordinates $q_1, \dots, q_4, p_1, \dots, p_4$. The base locus is defined by equations~$\sigma_1=\sigma_k=0$. (Again $k \in \left\{2,\dots,33\right\} \setminus \left\{2,14,15,33\right\}$.)  Denote the coordinates in the base locus the following way: $Q_1=\sigma_1$, $P_1=\sigma_k$, $z_3=q_2$, $z_4=q_3$, $z_5=q_4$, $z_6=p_2$, $z_7=p_3$, $z_8=p_4$.

Let $[l_1,l_2]$ be homogeneous coordinates in $\mathbb{CP}^1$. Our aim is to determine these coordinates. Denote the blow up by $\overset{\sim}{\Omega}$:
\begin{equation*}
	\overset{\sim}{\Omega} = \left\{ \left( (Q_1, P_1, z_3, \dots, z_8),\left[ l_1,l_2 \right] \right): Q_1 l_2 = P_1 l_1 \right\}.
\end{equation*}
The projection to the first coordinate is $\pi: \overset{\sim}{\Omega} \rightarrow \Omega$ and the exceptional divisor is 
$E=\pi^{-1}\left( (V_1 \cap \hat{V})^0\right)$.

Introduce two local charts $U_1$ and $U_2$ on $\overset{\sim}{\Omega}$. The charts $U_1$ and $U_2$ are given by $l_1\neq 0$ and $l_2 \neq 0$ respectively. The coordinates on chart $U_1$ are
\begin{alignat*}{3}
	x_1 &= Q_1, \mbox{\hspace{1cm}} &\\
	x_2 &= \frac{l_2}{l_1}=\frac{P_1}{Q_1}, \mbox{\hspace{1cm}} &\\
	x_i &= z_i, \mbox{\hspace{1cm}} & i=3, \dots, 8.
\end{alignat*}
The coordinates on chart $U_2$ are
\begin{alignat*}{3}
	y_1 &= \frac{l_1}{l_2}=\frac{Q_1}{P_1}, \mbox{\hspace{1cm}} &\\
	y_2 &= P_1, \mbox{\hspace{1cm}} &\\
	y_i &= z_i, \mbox{\hspace{1cm}} & i=3, \dots, 8.
\end{alignat*}
Obviously, the coordinate systems on the charts $U_1$ and $U_2$ are compatible and the transition functions are easy to work out.

Note that, the above ratio $\frac{P_1}{Q_1}$ determines the solution of our System~$(T)$ in the variable $H_{k-2}$ or $I_{k-15}$ due to Cramer's rule as long as $Q_1 \neq 0$ (i. e. outside the variety $V_1$): 
\begin{align*}
	H_{k-2} &= \frac{\det M_k}{\det M_1} = \frac{\sigma_k}{\sigma_1}=\frac{P_1}{Q_1}, \mbox{~if~} k=3, \dots,13; \\
	I_{k-15} &= \frac{\det M_k}{\det M_1} = \frac{\sigma_k}{\sigma_1}=\frac{P_1}{Q_1}, \mbox{~if~} k=16, \dots,32.
\end{align*}

However, the ratio $\frac{P_1}{Q_1}$ is undefined in $(V_1 \cap \hat{V})^0$. 
In other words, $H_{k-2}$ or $I_{k-15}$ is not determined by the coordinates $\left(Q_1=0, P_1=0, z_3, \dots, z_8\right)$, 
hence these will be the free parameter in the solution of System~$(T)$. 
If we fix the value of the free parameter on the exceptional divisor $E$, 
then we have a solution on $(V_1 \cap \hat{V})^0$. The coordinates on the charts $U_1$ and $U_2$ are:
\begin{align*}
	x_2 &= \frac{l_2}{l_1} = \frac{H_{k-2}}{1}, \\
	y_1 &= \frac{l_1}{l_2} = \frac{1}{H_{k-2}},
\end{align*}
if k=3, \dots,13. The case $k=16, \dots,32$ is similarly computable.

The coordinate $y_1$ is not equal zero in the chart $U_2$ due to $H_{k-2},I_{k-15} \neq \infty$, because $H_{k-2}$ and $I_{k-15}$ is a solution of the system of linear equations. 
The rank of the augmented matrix is greater than the rank of the coefficient matrix in the set $V_1 \setminus \hat{V}$. 
Consequently, we need to subtract the proper transform of $V_1$ from the blowing up. If the varieties $V_1$ and $\hat{V}$ are generically transverse, then the proper transform of $V_1$ will be the following:
\begin{equation*}
	\overset{\sim}{V_1} =\left\{ \left( \left( 0, P_1, q_2,\dots, p_4 \right), \left[ 0, 1 \right] \right) \in \Mod \times \mathbb{CP}^1 \right\}.
\end{equation*}

We have made numerical example for Conjecture~\ref{cj:main_V2} and made a blow up on $\left(T^{*} \left( \mathbb{C}\setminus P \right)^4 \setminus \Delta \right)/\Sym(4)$ along $(V_1 \cap \hat{V})^0$. We have gotten a family of solutions of systems of linear equations for coefficients of Equation~\eqref{eq:diffeq} parametrized by $H_{k-2}$ or $I_{k-15}$ in $(V_1 \cap \hat{V})^0$. 
Finally, we can associate a logarithmic connection to a fixed value $x_2 = \frac{l_2}{l_1} = \frac{H_{k-2}}{1}$ over the point $(0, 0, z_3, \dots, z_8) \in\Omega$. Hence the resulting set $\overset{\sim}{\Omega}\setminus \overset{\sim}{V_1}$ will be the part of the moduli space $\Mod$.

\section{The third order system with arbitrary number of parabolic points}
\label{sec:m=3}
\subsection{The existence of the variety}
We return to the description of the general case where we interrupted in Section~\ref{sec:third_order}. Hence we fix $n+1$ parabolic points on the Riemann surface $\Sigma$ and we fix exponents $\rho_{i,k}$ $i=0,\dots,n$, $k=1,2,3$, we have $N=3n-5$ apparent singularities and the same number of free parameters. In this section we will prove Theorem~\ref{th:main_V1} with the generalization of the method in Subsection~\ref{sec:M1}. We will not compute the discriminant variety exactly, we will only prove its existence. 
The aim is to show that the coefficient which belongs to the highest order exponent of $q_1$ in the polynomial associated to the variety is not zero.

As we saw in Proposition~\ref{prop:G_matrix}, the first $4n-4$ rows of the coefficient matrix are independent of the other rows. Denote the coefficient matrix of System~$(T)$ in the case of general $n$ by $M$.

Compute the determinant of $M$ with splits and row interchanges along the $3n-5$ common rows. Make confluent Vandermonde matrices as in Subsection~\ref{sec:M1}. Generalize the notation for $R_j$ and $S_j$ the following way: $J$ will be a multi-index which contains an index $j$ if $q_j$ appears in the last $n-2$ rows in the upper left $(8n-11) \times (8n-11)$ block. Hence $R_J$ is an $(8n-11) \times (8n-11)$ matrix where the last $n-2$ rows contain parameters $q_{j_1}, q_{j_2}, \dots, q_{j_{n-2}}$, where $j_1 < j_2 < \dots <j_{n-2}$ and $j_l \in J$ for all $l=1, \dots, n-2$. Similarly, $J$ in $S_J$ refers to the set $\overline{J}:=\{1, 2, \dots, 3n-5 \} \setminus J$. 
Through out this section we will use letter $j$ for index from $J$ and letter $\jj$ for index from $\overline{J}$, furthermore $i=1,\dots,n$.
 
Two splittings of $M$ are different if their indices $J_1$ and $J_2$ are not the same. On the other hand, all possible multi-indices $J$ appear in the expansion of the determinant of $M$. This shows that the expansion has $\binom{3n-5}{n-2}$ pairs of $(R_J,S_J)$. One of these pairs is the diagonal block of a matrix which has a zero block in the upper right or the lower left corner.

It is enough to consider the pairs which contribute to the coefficient of the highest exponent of $q_1$.

\subsection{The highest exponent}
\label{sec:highest_q1_exp}
First, we want to decide when we get higher $q_1$ power: if the common row with $q_1$ is in $R_J$ or is in $S_J$. 
The quantity and the order of parameters in $R_J$: $n$ type $t_i^{(0)}$, $2n-3$ type $q_k^{(1)}$ and $n-2$ type $q_j^{(2)}$. 
The quantity and the order of parameters in $S_J$: $n$ type $t_i^{(0)}$, $n-2$ type $q_j^{(2)}$ and $2n-3$ type $q_k^{(3)}$. 

The first $2n-3$ rows of the matrices $S_J$ come from Equation~\eqref{eq:hqj-iqj} and last $3n-5$ rows come from Equation~\eqref{eq:iqj_third}. The determinant does not change if we eliminate the coefficients related to the second derivatives of the $I(q_k)$'s from $2n-3$ rows in the last $3n-5$ rows. 
The remaining $n-2$ rows from the last $3n-5$ rows contain $q_j$'s with $j \in J$. Split $S_J$ along these $n-2$ rows such that one of them contains only the coefficients related to the second derivatives of the $I(q_j)$'s (denoted by $S'_J$), others contain one more coefficient related to the third derivative of $I(q_j)$ thus $\deg_{(q_1-q_k)} (S'_J)$ is greater than degree of any other matrix which comes from the splitting of $S_J$.

We have to count the degree of $q_1$ in the factors $\mu_1$, $\nu_1$, $\omega_1$, $\mu_l$, $\nu_l$ and $\omega_l$ (see Equations~\eqref{eq:mu-nu} and \eqref{eq:omega}): $\deg_{q_1}(\mu_1) = -2 n - 2(3 n - 6)$, $\deg_{q_1}(\mu_{l}) = -2$, $\deg_{q_1}(\nu_1) =\deg_{q_1}(\omega_1) = -3 n - 3(3 n - 6)$ and $\deg_{q_1}(\nu_{l}) =\deg_{q_1}(\omega_{l}) = -3$ where $l=2,\dots,3n-5$.

Now, compute the exponent of $q_1$ in the matrix $M$ with Formula~\eqref{eq:conf_vand}. The next computations build on the following (everywhere $j \in J\setminus\{1\}$, $\jj \in \comp{J}\setminus\{1\}$): the first $3$~terms count the exponents of $(q_1-q_j)$, $(q_1-q_{\jj})$ and $(q_1-t_i)$ in $R_J$; the fourth term counts the exponents of $\mu_1$ and $\mu_j$'s; the 5th--7th terms count $(q_1-q_j)$, $(q_1-q_{\jj})$ and $(q_1-t_i)$ in $S_J$; the 8th term counts $\nu_{\jj}$'s; the 9th term counts $\omega_j$'s; the 10th term counts $\omega_1$ in the first sum or two $\nu_1$'s in the second sum.

\noindent If the common row with $q_1$ is in $R_J$:
\begin{align*}
(3\cdot3)(n - 3) + (3\cdot2)(2 n - 3) + (3\cdot1) n - [2 n + 2(3 n - 6) + 2(n - 3)]+&\\ 
+(3\cdot4)(2 n - 3) + (3\cdot3) (n - 3) + (3\cdot1) n - 3 \cdot2(2 n - 3) - 3 (n-3) -&\\
- [3 n + 3(3 n - 6)] =&\\
=& 23 n - 45.
\end{align*}
\noindent If the common row with $q_1$ is in $S'_J$:
\begin{align*}
(2\cdot3)(n - 2) + (2\cdot2)(2 n - 4) + (2\cdot1) n - 2(n - 2) +(4\cdot4)(2 n - 4)+&\\
 + (4\cdot3) (n - 2) + (4\cdot1)n - 3\cdot 2(2 n - 4) - 3(n - 2) -& \\
 - 2\cdot[3 n + 3(3 n - 6)]=&\\
=& 23 n - 46.
\end{align*}
Hence the contribution of the common row with $q_1$ must be counted in the matrix $R_J$, i. e. $1 \in J$.

\subsection{The coefficient of the highest exponent}
The number of matrices which contribute to the coefficient of the highest $q_1$ exponent is the same as the number of sets $J$ with condition $1 \in J$, namely this cardinality is $\binom{3n-6}{n-3}$.

First, we compute the determinant of a matrix $R_J$ where $J=1, \dots, n-2$ (denote this set by $J_1$), i. e. the last $n-2$ rows of $R_J$ contain powers of $q_1, q_2, \dots, q_{n-2}$ respectively.
Matrix $R_{J_1}$ is a confluent Vandermonde matrix except for the factors $\mu_j$ ($j\in J_1$). 
Denote the number of inversions of parameters in $R_{J_1}$ by $\iota_r$ (see Definition~\ref{def:inversion}).
We do not compute $\iota_r$, because this would give a sign only.
\begin{align*}
	\det R_{J_1} &= r_{J_1} \cdot \prod_{j\in J_1} \mu_j := \\
	& = (-1)^{\iota_r} \mathrm{ConfVand}\left(t_1^{(0)},\dots,t_n^{(0)}, q_1^{(2)},\dots,q_{n-2}^{(2)}, q_{n-1}^{(1)},\dots,q_{3n-5}^{(1)}\right) \prod_{j\in J_1} \mu_j.
\end{align*}

The pair of $R_{J_1}$ is $S'_{J_1}$ which is a confluent Vandermonde, also except some factors $\nu_{\jj}$ and $\omega_j$. Denote the number of inversions by $\iota_s$.
\begin{align*}
	\det S'_{J_1} =& s_{J_1} \cdot \prod_{j\in J_1} (p_j \nu_j+\tilde{\nu}_j) \prod_{\jj\in\overline{J_1}} \nu_{\jj}^2 := \\
	=& (-1)^{\iota_s} \mathrm{ConfVand}\left(t_1^{(0)},\dots,t_n^{(0)}, q_1^{(2)},\dots,q_{n-2}^{(2)}, q_{n-1}^{(3)},\dots,q_{3n-5}^{(3)}\right) \cdot \\
	& \cdot \prod_{j\in J_1} (p_j \nu_j+\tilde{\nu}_j) \prod_{\jj\in\overline{J_1}} \nu_{\jj}^2.
\end{align*}

Produce the other terms of the expansion of the determinant of $M$ with index interchanges, as in Subsection~\ref{sec:M1}. Change elements between the set $J_1$ and the set $\comp{J_1}$. There is an important property: the index interchange does not change the parity of the number of inversions. 
Namely, if we adjust the sequence of indices $(J, \comp{J})$ to the standard sequence in each $3n-5$ row which contain $q_l$ ($l=1,\dots,3n-5$), then we change six times each. The upshot is that the number of row interchanges is even. 

All in all, the highest $q_1$ exponent is $23n-45$ due to the computation in Subsection~\ref{sec:highest_q1_exp} and the product which contains the highest exponent is the following:
\begin{equation}
\label{eq:q1_exponent}
\begin{split}
	&\sum_{\substack{J=\left\{j_1, \dots, j_{n-2} \right\} \\ 1 = j_1 < j_2 < \dots <j_{n-2} \leq 3n-5}} r_J s_J \prod_{j\in J} \mu_j (p_j \nu_j+\tilde{\nu}_j) \prod_{\jj\in\overline{J}} \nu_{\jj}^2 = \\
	=&\sum_{J} r_J s_J \prod_{j\in J} \mu_j \nu_j p_j \prod_{\jj\in\overline{J}} \nu_{\jj}^2 + \sum_{J} r_J s_J \prod_{j\in J} \mu_j \tilde{\nu}_j \prod_{\jj\in\overline{J}} \nu_{\jj}^2.
\end{split}
\end{equation}
If $j\neq 1$ then $\deg_{q_1} (\nu_j)=9-3n > \deg_{q_1} (\tilde{\nu}_j)=8-3n$ and if $j=1$ then $\deg_{q_1} (\nu_j)=18-12n > \deg_{q_1} (\tilde{\nu}_j)=17-12n$ according to Equation~\eqref{eq:mu-nu}. Hence the coefficient of the highest $q_1$ exponent contains the factor $\prod_{j\in J} p_j$. Each polynomials in the first sum in the second row of the Equation~\eqref{eq:q1_exponent} have different $\prod_{j\in J} p_j$ factors, hence the first sum does not vanish.

\begin{proof}[Proof of Theorem~\ref{th:main_V1}]
The above discussion shows that there exists a nonzero $\Sym(N)$ invariant polynomial with coordinates $\{q_j, p_j\}_{j=1}^{j=N}$. The polynomial defines an affine subvariety $V$ in \mbox{$T^{*}(\mathbb{C} \setminus P)^N$} and a part of \mbox{$T^{*}(\mathbb{C} \setminus P)^N$} outside $V$ is a dense open set of the moduli space $\Mod$.
\end{proof}

\bibliography{log_bib_Ivanics}

\begin{thebibliography}{10}

\bibitem{anosov}
D.~Anosov and A.~Bolibruch.
\newblock {\em The Riemann-Hilbert problem}.
\newblock Friedrich Vieweg \& Sohn Verlagsgesellschaft, 1994.

\bibitem{dubrovin}
B.~Dubrovin and M.~Mazzocco.
\newblock Canonical structure and symmetries of the {S}chlesinger equations.
\newblock {\em Communications in Mathematical Physics}, 271(2):289--373, 2007.

\bibitem{ha-gibson}
T.~T. Ha and J.~A. Gibson.
\newblock A note on the determinant of a functional confluent {V}andermonde
  matrix and controllability.
\newblock {\em Linear Algebra Appl.}, 30:69--75, 1980.

\bibitem{horn}
R.~A. Horn and C.~R. Johnson.
\newblock {\em Topics in matrix analysis}.
\newblock Cambridge University Press, 1994.

\bibitem{iena}
O.~Iena and A.~Leytem.
\newblock On the singular sheaves in the fine {S}impson moduli spaces of
  1-dimensional sheaves.
\newblock {\em Canadian Mathematical Bulletin}, 60:522--535, 2017.

\bibitem{inaba}
M-A. Inaba.
\newblock Moduli of parabolic connections on curves and the {R}iemann-{H}ilbert
  correspondence.
\newblock {\em Journal of Algebraic Geometry}, 22:407--480, 2013.

\bibitem{katz}
N.~Katz.
\newblock Rigid local systems.
\newblock {\em Ann. Math. Studies}, 139:563--626, 1996.

\bibitem{put}
M.~Put and M.~Singer.
\newblock {\em Galois Theory of Linear Differential Equations}.
\newblock Springer-Verlag, 2003.

\bibitem{szilard08}
Sz. Szab\'o.
\newblock The extension of a {F}uchsian equation onto the complex line.
\newblock {\em Acta Scientiarum Mathematicarum}, 74:557--564, 2008.

\bibitem{szilard13}
Sz. Szab\'o.
\newblock The dimension of the space of {G}arnier equations with fixed locus of
  apparent singularities.
\newblock {\em Acta Scientiarum Mathematicarum}, 79:107--128, 2013.

\end{thebibliography}
\bibliographystyle{plain}
\end{document}